\documentclass[12pt]{article}
\usepackage[english]{babel}

\usepackage{amsmath, amsthm}
\usepackage{amscd}
\usepackage{amsfonts}
\usepackage{amssymb}
\usepackage{tikz-cd}
\usepackage{mathabx}
\usepackage{mathtools}
\usepackage{graphicx}
\usepackage{hyperref}
\usepackage{epigraph}
\usepackage{underscore}

\usepackage[left=2cm,right=2cm,
    top=2cm,bottom=2cm,bindingoffset=0cm]{geometry}

\title{Weil Conjectures Exposition}
\author{Evgeny Goncharov\\ \href{mailto:eg555@cam.ac.uk}{eg555@cam.ac.uk}\footnote{Cambridge University, Center for Mathematical Sciences, Wilberforce Road, Cambridge, UK}, \href{mailto:eagoncharov@edu.hse.ru}{eagoncharov@edu.hse.ru}\footnote{National Research University Higher School of Economics (NRU HSE), Usacheva 6, Moscow, Russia.}}
\date{}

\newtheorem{fact}{Proposition}

\newtheorem{cor}{Corollary}

\newtheorem*{defn}{Definition}

\newtheorem{lemma}{Lemma}

\newtheorem{remark}{Remark}

\newtheorem*{thrm1}{Theorem (Grothendieck)}

\newtheorem*{thrm2}{Theorem (Inverse Mapping Theorem)}

\newtheorem*{thrm3}{Theorem (Grothendieck existence theorem)}

\newtheorem*{thrm4}{Theorem (Leray spectral sequence)}

\newtheorem*{thrm5}{Theorem (Properties of $l$-adic cohomology)}

\newtheorem*{thrm6}{Theorem (Cohomological interpretation of the zeta function)}

\newtheorem*{thrm7}{Theorem (Cohomological interpretation of the $L$-functions)}

\newtheorem*{thrm8}{Theorem (Poincare duality, first form)}

\newtheorem*{thrm9}{Theorem (Poincare duality, second form)}

\newtheorem*{thrm10}{Theorem (Poincare duality, third form)}

\newtheorem*{thrm11}{Deligne's Theorem (DT)}

\newtheorem*{thrm12}{Theorem (Weak Lefschetz)}

\newtheorem*{thrm13}{Theorem (The Main Lemma)}

\newtheorem*{thrm14}{Theorem (Kajdan-Margulis)}

\newtheorem*{thrm15}{Deligne's theorem (particular case)}

\newtheorem*{thrm16}{The rationality theorem}

\newtheorem*{thrm17}{Generalization of the Hasse-Weil bound}

\newtheorem*{thrm18}{Ramunajan Conjecture}

\newtheorem*{thrm19}{Ramunajan-Peterson Conjecture}

\newtheorem*{thrm20}{Bound on the character sums in several variables}

\newtheorem*{thrm21}{The hard Lefschetz theorem}

\newtheorem*{reduction1}{Reduction 1}

\newtheorem*{reduction2}{Reduction 2}

\newtheorem*{dt}{Theorem (DR)}

\begin{document}

\maketitle

\begin{abstract}
In this paper we provide a full account of the Weil conjectures including Deligne's proof of the conjecture about the eigenvalues of the Frobenius endomorphism. 

Section 1 is an introduction into the subject. Our exposition heavily relies on the Etale Cohomology theory of Grothendieck so I included an overview in Section 2. Once one verifies (or takes for granted) the results therein, proofs of most of the Weil conjectures are straightforward as we show in Section 3. 

Sections 4-8 constitute the proof of the remaining conjecture. The exposition is mostly similar to that of Deligne in [7] though I tried to provide more details whenever necessary. Following Deligne, I included an overview of Lefschetz theory (that is crucial for the proof) in Section 6. 

Section 9 contains a (somewhat random and far from complete) account of the consequences. Numerous references are mentioned throughout the paper as well as briefly discussed in Subsection 1.4. 

\end{abstract}

\tableofcontents

\section{Preliminaries}

In this section we formulate the Weil conjectures, state a few examples and outline how the results have evolved over the time. We finish by discussing the references.

\subsection{Statement of the Weil conjectures} 

Let $X_0$ be a nonsingular (= smooth) projective variety over $\mathbb{F}_q$ (a field with $q=p^a$, $p$-prime, $a \in \mathbb{Z}_{>0}$ elements)  of dimension $d$. We let $X$ be the variety obtained from $X_0$ by extension of scalars of $\mathbb{F}_q$ to $\mathbb{\bar F}_q$ and $X_0(\mathbb{F}_{q^n})$ be the set of points of $X_0$ with coordinates in  $\mathbb{F}_{q^n}$. For a set $A$ we denote by $\# A$ the number of elements in $A$ (so $\# X_0(\mathbb{F}_{q^n})$ is the number of points of $X_0$ with coordinates in $\mathbb{F}_{q^n}$). 

\begin{defn}
We define the zeta function of $X_0$ to be $$Z(X_0, t)=\exp \left( \sum_{n \geq 1} \# X_0(\mathbb{F}_{q^n}) \frac{t^n}{n} \right).$$
\end{defn}
Thus, $Z(X_0, t) \in \mathbb{Q}[[t]]$ is a formal power series with rational coefficients. Note that we have $$t \frac{d}{dt}\log Z(X_0, t)= \sum_{n \geq 1} \# X_0(\mathbb{F}_{q^n}) t^n$$ so $t \frac{d}{dt}\log Z(X_0, t)$ is the generating function for the sequence $\{ \# X_0(\mathbb{F}_{q^n}) \}_{n \geq 1}$. 

We now state the \textbf{Weil conjectures} (the original statements can be found in Weil's paper [1] published in 1949):

\textbf{(I) Rationality}: $Z(X_0, t)$ is a \textit{rational} function of t. Moreover, we have $$Z(X_0, t)=\frac{P_1(X_0, t)P_3(X_0, t) \cdots P_{2d-1}(t)}{P_0(X_0, t) P_2(X_0, t) \cdots P_{2d}(X_0, t)}$$ where $P_0(X_0, t)=1-t$, $P_{2d}(X_0, t)=1-q^d t$ and each $P_i(X_0, t)$ is an \textit{integral} polynomial.

\textbf{(II) Functional equation} $Z(X_0, t)$ satisfies the \textit{functional equation} $$Z(X_0, q^{-d} t^{-1})=\pm q^{d \chi/2}t^{\chi}Z(X_0, t),$$ where $\chi=\sum_i (-1)^i \beta_i$ for $\beta_i=\deg P_i(t)$. 

\textbf{(III) Betti numbers} If $X$ lifts to a variety $X_1$ in characteristic $0$, then $\beta_i$ are the (real) \textit{Betti numbers} of $X_1$ considered as a variety over $\mathbb{C}$.

\textbf{(IV) Riemann hypothesis} For $1 \leq i \leq 2d-1$,  $P_i(t)=\prod_{j=1}^{\beta_i} (1-\alpha_{i,j}t)$, where $\alpha_{i,j}$ are algebraic integers of absolute value $q^{i/2}$. 

Although Weil does not explicitly say so, it is clear that the conjectures are partially suggested by topological statements (in particular, (I) is suggested by the \textit{Lefschetz fixed-point formula} and (II) is suggested by \textit{Poincare duality}). Note that $\chi$ in (II) is the \textit{Euler characteristic} of $X$ (which can be defined as the \textit{self-intersection number} of the diagonal with itself in $X \times X$) and (III) suggests the existence of a cohomological theory such that $\beta_i$ are the Betti numbers (dimensions of the cohomology groups) of $X$ in that theory. 

If such a theory also had an analog of the Lefschetz fixed-point formula and of Poincare duality (together with a theorem relating the Betti numbers of $X$ and $X_1$), then proofs of most of the conjectures (except for the Riemann hypothesis) would be straightforward. A search for such a theory (a "\textit{Weil cohomology}" theory) eventually resulted in the invention of \textit{Etale Cohomology} (and \textit{$l$-adic cohomology}) by Artin, Verdier and Grothendieck. In terms of this $l$-adic cohomology, the "moral" of (I)-(IV) is that topological properties of $X_0$ determine the diophantine shape of its zeta function. 

\subsection{A different definition of the zeta function and the Frobenius endomorphism}

To explain the connection between (IV) above and the original Riemann hypothesis (for the \textit{Riemann zeta function}) we need to make an alternative definition of the zeta function (it is also the definition used by Deligne in [7]):

\begin{defn}
Let $Y$ be a scheme of finite type over $Spec (\mathbb{Z})$, $|Y|$ the set of closed points of $Y$ and for $y \in |Y|$ denote by $N(y)$ the number of elements in the residue field $k(y)$ of $Y$ at $y$. The Hasse-Weil zeta function of $Y$ is $$\zeta_Y(s)=\prod_{y \in |Y|}(1-N(y)^{-s})^{-1}$$ (this product absolutely converges for $Re(s) > \dim Y$). 
\end{defn}

Note that for $Y=Spec(\mathbb{Z})$, $\zeta_Y(s)$ is the Riemann zeta function (we will elaborate more on this in the last few paragraphs of this paper). 

We can regard a variety $X_0$ as a \textit{scheme of finite type} over $\mathbb{Z}$ via $$X_0 \to Spec(\mathbb{F}_q) \hookrightarrow Spec(\mathbb{Z})$$ (where the second map is defined by $\mathbb{Z} \to \mathbb{Z}/p\mathbb{Z} \hookrightarrow \mathbb{F}_q$). For $x \in X_0$ let $\deg(x)=[k(x) : \mathbb{F}_q]$. We have $N(x)=q^{\deg(x)}$ so it makes sense to introduce a new variable $t=q^{-s}$and let $$Z(X_0, t)=\prod_{x \in |X_0|}(1-t^{\deg(x)})^{-1}.$$ This product converges for $|t|$ small enough and we have $$\zeta_{X_0}(s)=Z(X_0, q^{-s}).$$

Given that this definition of $Z(X_0, t)$ coincides with the one given in Subsection 1.1, the Riemann hypothesis for $X_0$ states that $\zeta_{X_0}(s)$ has its poles on the lines $Re(s)=0, 1, \cdots, \dim X_0$ and zeroes on the lines $Re(s)=\frac{1}{2}, \frac{3}{2}, \cdots, \frac{\dim X_0 - 1}{2}$, so the analogy with the original Riemann hypothesis becomes clear. 

Let me now introduce the \textit{Frobenius endomorphism}\footnote{Formally, we do not need to introduce the Frobenius endomorphism at this point and all we need here is formula $(1)$.} $F$ (that will play a central role in our exposition) and explore the connection between the fixed points of $X$ under $F^n, \ n \geq 1$ and the varieties $X_0(\mathbb{F}_{q^n})$. Define the Frobenius endomorphism as\footnote{For clarification of the definition and proof of the basic properties of $F$ see [2], Chapter 27.} $F: X \to X, \ x \to x^q$. By this we mean that on every open affine $U$ (defined by some $U_0$) $F$ acts by taking the $q$-th power of all the coordinates (there exists a unique such $F$). We have $F^{-1}(U)=U$ and $\deg F=q^{\dim X}$. Identify the set $|X|$ of closed points of $X$ with $X_0( \mathbb{\bar F}_q)$ (those are the points $Hom_{\mathbb{F}_q} (Spec(\mathbb{\bar F}_q), X_0)$ of $X_0$ with coefficients in $\mathbb{\bar F}_q$). Let $\varphi \in Gal(\mathbb{\bar F}_q/\mathbb{F}_q)$ be the \textit{substitution of Frobenius}. The action of $F$ on $|X|$ identifies with the action of $\varphi$ on $X_0(\mathbb{\bar F}_q)$. Under this equivalence:

\textbf{(a)} The set $X^F$ of closed points of $X$ fixed under $F$ is identified with the set $X_0(\mathbb{F}_q) \subset X_0(\mathbb{\bar F}_q)$ of points of $X$ defined over $\mathbb{F}_q$ (just because for $x \in \mathbb{\bar F}_q$ one has $x \in \mathbb{F}_q \Leftrightarrow x^q=x$).

\textbf{(b)} Similarly, the set $X^{F^n}$ of closed points of $X$ fixed under the $n$-th iteration of $F$ is identified with $X_0(\mathbb{F}_{q^n})$.

\textbf{(c)} The set $|X|$ of closed points of $X$ is identified with the set $|X|_F$ of the orbits of $F$ (or $\varphi$) on $|X|$. The degree $\deg(x)$ of $x \in |X_0|$ is the number of elements in the corresponding orbit.

\textbf{(d)} From (b) and (c) we see that $$\#X^{F^n}=\#X_0(\mathbb{F}_{q^n})=\sum_{\deg(x)|n} \deg(x) \ \ \textbf{(1)}$$ (for $x \in |X_0|$ and $\deg(x)|n$, $x$ defines $\deg(x)$ points with coordinates in $\mathbb{F}_{q^n}$ all conjugate over $\mathbb{F}_q$). We will occasionally use these results.

Once we have $(1)$, we can write $$t \frac{d}{dt}\log Z(X_0, t)=\frac{t\frac{d}{dt}Z(X_0, t)}{Z(X_0, t)}=\sum_{x \in |X_0|} -\frac{-\deg(x)t^{\deg(x)}}{1-t^{\deg(x)}}=$$ $$=\sum_{x \in |X_0|} \sum_{n >0} \deg(x)t^{n\deg(x)} \overset{(1)} = \sum_n X_0(\mathbb{F}_{q^n})t^n$$ so the equivalence of two definitions for the zeta function follows by taking exponents.

\subsection{Examples and evolution of the results}

It is easy to verify the Weil conjectures by hand for $X_0=\mathbb{P}_{\mathbb{F}_q}^d$. Indeed, since in this case we have $X_0(\mathbb{F}_{q^n})=\sum_{i=0}^d q^{in}$, the zeta function is $\prod_{i=0}^n (1-q^{i}t)^{-1}$ ($\chi=d+1$ and the complex projective space gives the relevant Betti numbers $\beta_i$). In general, it is easy to check the Weil conjectures for varieties that are disjoint unions of a finite number of affine spaces, since their zeta functions are easy to find. These include Grassmannians and flag varieties.  

Of course, curves (the original motivation of Weil) are a perfect example to illustrate the Weil conjectures. I highly recommend the lectures\footnote{I would also like to thank Richard Griffon for inspiring my interest in the Weil Conjectures.} [3] by Richard Griffon (available online) for a great account of the matter. 

Suppose, therefore, that $X_0$ is a smooth projective curve of genus $g$ over $\mathbb{F}_q$. One can use the Riemann-Roch theorem to show that $Z(X_0, t)$ is a rational function of $t=q^{-s}$ that has the form $$Z(X_0, t)=\frac{P(t)}{(1-t)(1-qt)},$$ where $P(t)=\prod_{i=1}^{2g} (1-\alpha_i t)$ is a polynomial of degree $2g$ with integer coefficients and its roots are permuted by $\alpha_i \to q/\alpha_i$ (this implies the functional equation by a simple calculation - see section 3 (II) for a generalization of this argument to arbitrary dimension). The original proof of the above statements is a result due to Schmidt [4]. 

The Riemann Hypothesis for curves states that the zeroes of $\zeta_{X_0}(s)$ lie on the line $Re(s)=\frac{1}{2}$, or, equivalently, that $|\alpha_i|=\sqrt{q}$. Taking the logarithms of both sides of the equality $$\exp \left( \sum_{n \geq 1} \# X_0(\mathbb{F}_{q^n}) \frac{t^n}{n} \right)=\frac{\prod_{i=1}^{2g} (1-\alpha_i t)}{(1-t)(1-qt)}$$ we obtain $$\# X_0(\mathbb{F}_{q^n})=1+q^n-\sum_{i=1}^{2g} \alpha_i^n \ \ \textbf{(2)}.$$ Then (as pointed by Hasse in [5]) the Riemann Hypothesis becomes \textbf{equivalent} to the following diophantine statement $$|\# X_0(\mathbb{F}_{q^n})-1-q^n| \leq 2g \sqrt{q}^n \ \textbf{(Hasse-Weil bound)}.$$

Hasse went on to provide two independent proofs for the case of elliptic curves $(g=1)$ in 1933 and 1934 based on lifting a curve to characteristic 0 and studying the endomorphism ring of an elliptic curve respectively.

Then Weil himself proved the general case in two different ways, one of which he then generalized to Abelian varieties (see the extended overview of the methods of Hasse and Weil along with references to the original papers in [6]). 

It was not much before Deligne's proof of the Riemann Hypothesis for arbitrary varieties (1973) [7] that an elementary proof for the case of curves was found by Stepanov [8] and further simplified by Bombieri [9]. Stepanov introduced an elementary method to upper-bound the quantity $|\# X_0(\mathbb{F}_{q^n})|$ (a version of the now-called polynomial method) and Bombieri applied a Galois-theoretic trick to establish a matching lower bound, that, after using (2) and the "tensor-power trick" implied the Hasse-Weil Bound. 

Before that came a completely unexpected (non-homological) Dwork's $p$-adic proof of the rationality conjecture [10]. 

Though Weil himself, apparently, did not initiate the search for a "good" cohomology theory while stating his conjectures, that search attributed for much of the development of algebraic geometry in the following few decades (again, see [6] for more details). The theory (or theories for every prime $l \neq p$), that was finally established by M.Artin, J.L.Verdier and A.Grothendieck in the early 1960s (and appeared in [11] a few years later) is based on the notion of an \textit{etale covering space} (we will review it in the next section) and they went on to prove most of the conjectures (as we shall do in Section 3).

Serre, following a suggestion of Weil, has formulated and proved the analog of the Weil conjectures for Kahler manifolds making essential use of the Hodge index theorem. Inspired by this result, Grothendieck formulated a number of very strong \textit{existence} and \textit{positivity} conjectures about \textit{algebraic cycles} (published at [12] in 1969), which implied the Riemann hypothesis. 

Despite the common belief that the proof of the Riemann Hypothesis would require some of the standard conjectures, Deligne managed to avoid them altogether in his proof (and even managed to deduce one of them - the Hard Lefschetz theorem that we shall state in the last section). He later said (in the interview on the occasion of his 2013 Abel Prize) that he was just lucky to learn the theory of automorphic forms and stumble upon the work of Rankin [13], which helped him to deduce the \textit{Main Lemma} and handle the case of \textit{hypersurfaces of odd dimension} (see Section 5 and Subsection 6.3). 

Deligne went on to state and prove the generalization of the Weil conjectures in [14]. Much of the proof is the rearrangement of the first proof and the main extra ingredient is an analog of the classical argument of Hadamard and de la Vallee-Poussin used to show that various  L-series do not have zeroes with real part equal to 1 (also see Subsection 9.2). 

Laumon has found another proof using Deligne's \textit{$l$-adic Fourier transform} [15] in 1987 that was later further simplified by Katz [16] in 2001 by using monodromy results in the spirit of the first proof of Deligne. Finally, Kedlaya gave yet another proof using the Fourier transform but replacing $l$-adic cohomology by rigid cohomology [17] in 2006. 

I would like to finish this subsection by mentioning that there were attempts to find an "elementary" proof of the Riemann hypothesis by directly counting the number of rational points of a variety over finite extensions of $\mathbb{F}_q$ (like the Bombieri-Stepanov proof for the case of curves) but, as pointed out by Katz [6], it might not be available since the Riemann hypothesis does not seem to be equivalent to any diophantine statement. Indeed, from the rationality conjecture we get $$|X_0(\mathbb{F}_{q^n})|=\sum_{0 \leq i 
\leq 2d} (-1)^i \sum_j \alpha_{i, j}^n=q^{nd}+1+\sum_{1 \leq i 
\leq 2d-1} (-1)^i \sum_j \alpha_{i, j}^n$$ and the only obvious estimate for the error term from the Riemann hypothesis is $O(q^{n(d-1/2)})$. The proof for the cases when the Riemann hypothesis is equivalent to a diophantine statement (complete intersections, simply-connected surfaces, etc.) would certainly be interesting, but I am not aware of any results of this sort.

\subsection{References}

Our basic references for the next sections are Deligne's original paper [7] (or see my English translation - [33]) and Milne's lectures [2] (contains a detailed overview of Etale Cohomology theory and a partially incomplete proof of the Weil Conjectures). Most of the results that we will need were established in SGA 4, 5 and 7. See (2.14) and (5.13) in [7] for the bibliographical notes.

For a good overview of Etale cohomology consult SGA $4\frac{1}{2}$ by Deligne (sketches most of the required proofs as well as takes great care to relate the results with their classical analogs) or [2] together with the book [18] of Milne containing more material and explaining some of the skipped proofs. 

Reading the overview of Katz [6] can be beneficial as an introduction, an overview of related results and a detailed proof of the particular case of hypersurfaces of odd dimension. 

Some further references will appear along the way (especially in Sections 5 and 9). 

\section{Overview of Etale cohomology}

In the first subsection I will (very briefly) recall the basic definitions of the Etale Cohomology theory. It is not meant and is not an introduction into the subject (one really should read SGA $4\frac{1}{2}$ or [2] for that). In particular, we will usually not give any proofs.  
In Subsections 2 and 3 we discuss some of the more advanced concepts that will be crucial for the exposition. 

In Subsection 4 we state a theorem that combines many of the important results in $l$-adic cohomology. Once discussing these major results, we move on to give \textit{cohomological interpretation} of zeta and L-functions and discuss the Poincare duality for $l$-adic cohomology in detail (in Subsections 5 and 6 respectively).

One may use the results in Etale Cohomology appearing in this section (and later on) as a "black box" and still get some benefit from the proofs of Section 3 and Sections 4-7. The intuition to keep in mind here is that Etale cohomology with coefficients in a constant (\textit{torsion}) sheaf behaves "like singular cohomology"\footnote{In fact, for quasi-coherent sheaves it coincides with coherent cohomology.}.
You are free to skip 2.1-2.4 if you are already comfortable with $l$-adic cohomology.

\subsection{Basic definitions}

I assume the knowledge of ordinary sheaf cohomology. 

First, let me say a few words to justify why a complicated cohomology theory would even be necessary. The \textit{Zariski topology} is inadequate. There are two (and, in fact, many more) important reasons for that:

\textbf{(1)} The cohomology groups are (usually) zero: 
\begin{thrm1}
If $X$ is an irreducible topological space, then $H^r(X, \mathcal{F})=0$ for all constant sheaves $\mathcal{F}$ and all $r>0$. 
\end{thrm1}

\begin{proof}
Every open subset $U$ of $X$ is connected. Hence, if $\mathcal{F}$ is a constant sheaf defined by the group $\Lambda$, then we have $\mathcal{F}(U)=\Lambda$ for all open nonempty $U$. This implies that $\mathcal{F}$ is \textit{flabby} and so $H^r(X, \mathcal{F})=0$ for $r>0$.
\end{proof}

\textbf{(2)} The inverse mapping theorem does not hold:

A $C^{\infty}$ map of differentiable manifolds $\varphi:Y \to X$ is said to be \textit{etale} at $y \in Y$ if the map on tangent spaces $d\varphi: T_x Y \to T_{\varphi(x)} X$ is an isomorphism.

\begin{thrm2}
A $C^{\infty}$ map of differentiable manifolds is a local isomorphism at every point at which it is etale. 
\end{thrm2}

Now, if we define a regular map of nonsingular varieties $\varphi:Y \to X$ to be \textit{etale} at $y \in Y$ if the map on tangent spaces $d\varphi: T_x Y \to T_{\varphi(x)} X$ is an isomorphism, it is easy to see that $\mathbb{A}_k^1 \to \mathbb{A}_k^1, \ x \to x^n$, where $n$ does not divide the characteristic of $k$, is etale on the complement of the origin but is not a local isomorphism at any point $x \neq 0$ for $n > 1$ (since it would imply that the map of function fields $k(t) \to k(t), \ t \to t^n$ is an isomorphism). 

Perhaps Weil himself realized that there cannot be a "good" cohomology theory over $\mathbb{Q}$ because of the existence of \textit{supersingular elliptic curves} over finite fields. We sketch the argument as follows.

Let $p$ be the characteristic of the field $\mathbb{F}_q$ and $C$ be a supersingular elliptic curve over $\mathbb{F}_q$. The endomorphism ring of $C$ is an order in a \textit{quaternion algebra} over $\mathbb{Q}$ that should act on the first cohomology group - a 2-dimensional vector space over $\mathbb{Q}$. However, the quaternion algebra does not have such an action. Similarly, one eliminates the possibility that the coefficient ring is $\mathbb{R}$ or $\mathbb{Q}_p$ since the quaternion algebra is still a \textit{division algebra} over these fields. However, for a prime $l \neq p$ the division algebra splits over $\mathbb{Q}_l$ and becomes a 2-dimensional matrix algebra that can act on a 2-dimensional vector space. 

With these considerations in mind, we introduce the Etale Cohomology theory as follows (see [2] for all the details):

First, we say that a morphism $\varphi: Y \to X$ of schemes is \textit{etale} if it is \textit{flat} and \textit{unramified} (in particular, \textit{of finite type}). For an explanation and the proof that for the case of nonsingular varieties the two definitions coincide (together with a separate definition for arbitrary varieties using the notion fo \textit{tangent cones}) consult [2], Chapter 2. 

One does not need to have a topological space to build up a sheaf theory (and a cohomology theory for sheaves). Indeed, let \textbf{C} be a category with, for each object $U$ of \textbf{C} a distinguished set of families of maps $\{U_i \to U \}_{i \in I}$, called the \textit{coverings} of $U$, that satisfy:

\textbf{(a)} For a covering $\{U_i \to U \}_{i \in I}$ of $U$ and any morphism $V \to U$ in \textbf{C}, the fiber products $U_i \times_U V$ exist and $\{U_i \times_U V \to V \}_{i \in I}$ is a covering of $V$. 

\textbf{(b)} If $\{U_i \to U \}_{i \in I}$ is a covering of $U$ and for each $ i \in I$ $\{V_{ij} \to U_i \}_{j \in J_i}$ is a covering of $U_i$, then $\{V_{ij} \to U \}_{i, j}$ is a covering of $U$. 

\textbf{(c)} For all $U$ in \textbf{C} the family $\{ U \overset{id} \to U \}$ consisting of a single map is a covering of $U$. 

Such a system of coverings is called a \textit{Grothendieck topology} on \textbf{C} and \textbf{C} together with this topology is called a \textit{site}. 

We define the \textit{etale site of $X$} (denoted $X_{et}$) as a category $Et/X$ with objects the etale morphisms $U \to X$ and arrows the $X$-morphisms (the obvious commutative diagrams) $\varphi: U \to V$. The \textit{coverings} are surjective families of etale morphisms $\{U_i \to U \}_{i \in I}$ in $Et/X$.
An \textit{etale neighborhood} of $x \in X$ is an etale map $U \to X$ together with a point $u \in U$ mapping to $x \in X$. These give us the notion of \textit{etale topology} on $X$.

A \textit{presheaf} for the etale topology on $X$ is a contravariant functor $\mathcal{F}: Et/X \to Ab$. It is a \textit{sheaf} if the sequence $$\mathcal{F}(U) \to \prod_{i \in I} \mathcal{F}(U_i) \rightrightarrows \prod_{(i, j) \in I \times I} \mathcal{F}(U_i \times_{U} U_j)$$ is exact for all etale coverings $\{U_i \to U \}_{i \in I}$. 

One shows that the category of sheaves of abelian groups $Sh(X_{et})$ is an \textit{abelian category} with enough \textit{injectives} ([2], Chapter 8). The functor $$Sh(X_{et}) \to Ab, \ \mathcal{F} \to \Gamma(X, \mathcal{F})$$ is left exact and we can define $H^r(X_{et}, - )$ as its $r$-th \textit{right derived functor}. One then has the usual properties (that follow from the theory of derived functors):

\textbf{(a)} For any sheaf $\mathcal{F}$, $H^0(X_{et}, \mathcal{F})=\Gamma(X, \mathcal{F})$.

\textbf{(b)} $H^r(X_{et}, \mathcal{I})=0$ for $r>0$ if $\mathcal{I}$ is injective. 

\textbf{(c)} One can functorially associate to a \textit{short exact sequence} of sheaves $$0 \to \mathcal{F}' \to \mathcal{F} \to \mathcal{F}'' \to 0$$ a \textit{long exact sequence} in cohomology $$0 \to H^0(X_{et}, \mathcal{F}') \to H^0(X_{et}, \mathcal{F}) \to H^0(X_{et}, \mathcal{F}'') \to H^1(X_{et}, \mathcal{F}') \to \cdots .$$

Note that these properties characterize the functors $H^r(X_{et}, - )$ uniquely (up to a unique isomorphism). One may go on to prove the analogs of the classical axioms of cohomology (dimension, exactness, excision, homotopy - see [2], Chapter 9).  

As in the classical case, calculations are usually performed using the \textit{Cech cohomology groups} (see [2], Chapter 10). They coincide with the Etale cohomology groups for all $X$ such that $X$ is quasi-compact and every finite subset of $X$ is contained in an open affine (all quasi-projective varieties, per se). 

\subsection{Cohomology groups with compact support and $l$-adic cohomology}

For a morphism $\pi:Y \to X$ of schemes one can define the functor $\pi_*:Sh(Y_{et}) \to Sh(X_{et})$ by setting $$\pi_* \mathcal{F}(U) =\mathcal{F}(U \times_X Y)$$ for every sheaf $\mathcal{F}$ on $Y_{et}$ (it makes sense since $U \times_X Y \to Y$ is also etale). This functor is called the \textit{direct image functor} and it admits a left adjoint $\pi^*$ (\textit{the inverse image functor}) defined as follows. For a sheaf $\mathcal{F}$ on $X_{et}$ we let $\pi^* \mathcal{F}=a(\mathcal{F}')$ where for $V \to Y$ etale we have $\mathcal{F}'(V)=\varinjlim \mathcal{F}(U)$ where the direct limit is taken over the commutative diagrams $$\begin{tikzcd}
V \arrow{r}{a} \arrow{d}{b}
&U \arrow{d}{c}\\
Y \arrow{r}{d} &X
\end{tikzcd}$$ with $U \to X$ etale and $a(\mathcal{F}')$ denotes the (unique) sheaffication of $\mathcal{F}'$. See [2], Chapter 8 for the properties of these functors (we will need some of them later). 

We will have to use cohomology groups with \textit{compact support}. The definition for a torsion sheaf on a variety is as follows:
\begin{defn}
Let $F$ be a torsion sheaf on a variety $U$. The cohomology groups with compact support are $$H_c^r(U, \mathcal{F})=H^r(X, j_{!}\mathcal{F}),$$ where $X$ is any complete variety containing $U$ as a dense open subvariety and $j$ is the inclusion map (we denote by $j_{!}$ the extension by $0$ of the direct image functor $j_{*}$- see [2], Chapter 8). 
\end{defn}
Cohomology groups with compact support coincide with ordinary cohomology groups for complete varieties.

To check that cohomology groups with compact support are well-defined one needs to show that:

\textbf{(a)} Every variety admits a completion.

\textbf{(b)} Cohomology groups with compact support are independent of the completion.

See [2], Chapter 18 for the topological motivation of the above definition and proofs of both statements (the second proof requires the \textit{proper base-change theorem} of Chapter 17). 

We now aim to define $l$-adic cohomology and start by defining cohomology with coefficients in $\mathbb{Z}_l$ ($l \neq p$ - a prime number). The problem here is that etale cohomology groups with coefficients in non-torsion sheaves are anomalous. For example, using the ordinary definition we would always get $H^1(X_{et}, \mathbb{Z}_l)=0$ for any normal $X$. 

The way out is to define $$H^r(X_{et}, \mathbb{Z}_l)=\varprojlim H^r(X_{et}, \mathbb{Z}/l^n\mathbb{Z}).$$ Setting $$H^r(X_{et}, \mathbb{Q}_l)=H^r(X_{et}, \mathbb{Z}_l) \otimes_{\mathbb{Z}_l} \mathbb{Q}_l$$ gives us the notion of $l$-adic cohomology groups (the trick here is that \textit{cohomology does not commute with inverse limits} of sheaves). The definition for cohomology groups with compact support is similar. 

Later on (starting with Subsection 2.5) we are going to use "\textit{constructible $\mathbb{Q}_l$ sheaves}" and "\textit{locally-constant constructible $\mathbb{Q}_l$ sheaves}" as well as cohomology with coefficients in these sheaves. The definitions are as follows (compare to the properties stated in [7], 1.9):

\begin{defn}
A sheaf $\mathcal F$ on $X_{et}$ is constructible if:

(i) For every closed immersion $i: Z \hookrightarrow X$ with $Z$ irreducible there exists a nonempty open subset $U \subset Z$ such that $(i^* \mathcal{F})|U$ is locally constant.

(ii) $\mathcal{F}$ has finite stalks (defined as in the classical case). 

\end{defn}

\begin{remark}
One can show that every torsion sheaf is the union of its constructible subsheaves.
\end{remark}

The following definition is motivated by analogy with the definition of a finitely generated $\mathbb{Z}_l$-module:

\begin{defn}
A sheaf of $\mathcal{M} $ of $\mathbb{Z}_l$-modules on $X$ (or a constructible $\mathbb{Z}_l$-sheaf) is a family $\{ \mathcal{M}_n, f_{n+1}: \mathcal{M}_{n+1} \to \mathcal{M}_n \}$ such that

(i) Each $\mathcal{M}_n$ is a constructible sheaf of $\mathbb{Z}/l^n \mathbb{Z}$ modules.

(ii) The maps $f_{n+1}$ induce isomorphisms $\mathcal{M}_{n+1}/l^n \mathcal{M}_{n+1} \overset{\sim} \to \mathcal{M}_n$. 
\end{defn}
Then we define the cohomology groups as $$H^r(X_{et}, \mathcal{M})=\varprojlim H^r(X_{et}, \mathcal{M}_n).$$ A sheaf of $\mathbb{Q}_l$-spaces is then just a $\mathbb{Z}_l$-sheaf except for the fact that we let $$H^r(X_{et}, \mathcal{M})=\varprojlim H^r(X_{et}, \mathcal{M}_n) \otimes_{\mathbb{Z}_l} \mathbb{Q}_l.$$ Again, the definitions of the corresponding cohomology groups with compact support are similar. 

We will say that a sheaf $\mathcal{M}=\{ \mathcal{M}_n \}$ is \textit{locally constant} if all the sheaves $\mathcal{M}_n$ are locally constant\footnote{This is an abuse of terminology. A locally constant sheaf (french. constant-tordu) on $X$ need not become trivial on any etale covering of $X$.}. We will sometimes (mostly to justify that a certain locally constant $\mathbb{Q}_l$-sheaf is constant, but it is also important to keep in mind for other situations) refer to the following proposition.
\begin{fact}
Assume that $X$ is connected and let $\bar x$ be a geometric point of $X$. For $\mathcal{F}$ a locally constant $\mathbb{Q}_l$-sheaf on $X$, $\pi_1(X, \bar x)$ acts\footnote{The etale fundamental group- see [2], Chapter 3.} on the stalks $\mathcal{F}_{\bar x}$ and the map $\mathcal{F} \to \mathcal{F}_{\bar x}$ defines an equivalence of categories:

(locally constant $\mathbb{Q}_l$-sheaves on $X$) $\to$ (continuous actions of $\pi_1(X, \bar x)$ on $\mathbb{Q}_l$-vector spaces of finite dimension).
\end{fact}
From now on we will deal with constructible $\mathbb{Q}_l$-sheaves only and call them just $\mathbb{Q}_l$-sheaves.

See [2], Chapter 19 for more details on $\mathbb{Z}_l$ and $\mathbb{Q}_l$-sheaves. 

\subsection{Higher direct images of sheaves, the Leray spectral sequence}

We will have to consider \textit{higher direct images} of sheaves in our analysis. Let $\pi: Y \to X$ be a morphism of schemes. The direct image functor $\pi_*$ is left exact (since it has a left adjoint $\pi^*$) and we can consider its right derived functors $R^i \pi_*$. Then for a sheaf $\mathcal{F}$ we call the sheaves $R^i \pi_* \mathcal{F}$ the higher direct images of $\mathcal{F}$. For the properties of these sheaves we direct the reader to [2], Chapter 12. We will need a way to compute the stalks of $R^i \pi_* \mathcal{F}$ (the proof is elementary).
\begin{fact}
The stalk of $R^r \pi_* \mathcal{F}$ at a geometric point $\bar x \to X$ is $\varprojlim H^r (U \times_X Y, \mathcal{F})$  where the limit is over all the etale neighborhoods $(U, u)$ of $\bar x$. 
\end{fact}

In the next few paragraphs we outline the role of \textit{spectral sequences} in Etale Cohomology. See [19] for a great introduction to (abstract) spectral sequences. 

The following general result is  a powerful source for spectral sequences, including those appearing in Etale Cohomology theory (for example, we have a spectral sequence, relating the Etale Cohomology groups with their Cech analog).

\begin{thrm3}
Let \textbf{A}, \textbf{B} and \textbf{C} be abelian categories such that \textbf{A} and \textbf{B} have enough injectives. Let $F: \textbf{A} \to \textbf{B}$ and $G: \textbf{B} \to \textbf{C}$ be left exact functors such that $(R^i G)(FI)=0$ for $i>0$ if $I$ is injective (for example, it is true if $F$ preserves injectives). Then there exists a spectral sequence $$ E_2^{pq}=(R^p G)(R^q F)(A) \Rightarrow R^{p+q}(FG)(A).$$
\end{thrm3}

We will need the following consequence called the Leray spectral sequence.
\begin{thrm4}
Let $\pi: Y \to X$ be a morphism of schemes. Then for any sheaf $\mathcal{F}$ on $Y_{et}$ there is a spectral sequence $$E_2^{pq}=H^p(X_{et}, R^q \pi_*\mathcal{F}) \Rightarrow H^{p+q}(Y_{et}, \mathcal{F}).$$
\end{thrm4}

\begin{proof}
The Leray spectral sequence follows as a special case of Grothendieck existence theorem since the functors $\pi_*$ and $\Gamma(X, -)$ are left exact, their composition is $\Gamma(Y, -)$ and $\pi_*$ preserves injectives. 
\end{proof}

Note that the sequence above is exactly the same as the classical Leray spectral sequence in topology. The point (in both cases) is that there is a way to reconstruct the cohomology groups of $Y$ from the map $\pi$ and $X$. 

We will apply the results of this subsection to the constant sheaf $\mathbb{Q}_l$ in the later stages of the proof of the Riemann hypothesis. 

\subsection{Major results in Etale cohomology}

Now that we have defined $l$-adic cohomology, we can try to prove that it has the properties (similar to the topological statements) of a "good" cohomology theory. We state the results below (mostly) for $l$-adic cohomology groups for simplicity, though many are true for arbitrary (or constructible) sheaves. Usually, such results are first proved for torsion sheafs and then extended to the sheaf $\mathbb{Q}_l$ (and $\mathbb{Q}_l$-sheaves) by passing to formal limits. We also formulate all of the statements for a projective nonsingular variety $X$ but I will mention some of the possible generalizations. 

\begin{thrm5}
The $l$-adic cohomology groups $H^i(X, \mathbb{Q}_l)$ of a projective nonsingular variety $X$ purely of dimension $d$ (meaning that all the irreducible components of $X$ have dimension $d$) over an algebraically closed field $k$ of characteristic $p$ (possibly $p = 0$) satisfy the following properties (for each prime $l \neq p$):

\textbf{(a) Finiteness} For all $i \geq 0$  $H^i(X, \mathbb{Q}_l)=0$ is a finite-dimensional $\mathbb{Q}_l$ vector space (the same is true for $\mathbb{Q}_l$-sheaves).

\textbf{(b) Cohomological dimension} $H^i(X, \mathbb{Q}_l)=0$ for $i > 2d$ and $H^i(X, \mathbb{Q}_l)=0$ for $i>d$ if $X$ is affine (the same is true for $\mathbb{Q}_l$-sheaves).

\textbf{(c) Kunneth isomorphism theorem} For all $i \geq 0$ we have a canonical isomorphism $$H^i(X \times Y, \mathbb{Q}_l) \simeq \bigoplus_{k+l=i} H^k(X , \mathbb{Q}_l) \otimes_{\mathbb{Q}_l} H^l(Y , \mathbb{Q}_l)$$ induced by the cup-product. 

\textbf{(d) Poincare duality} There is an isomorphism $H^{2d}(X, \mathbb{Q}_l) \simeq \mathbb{Q}_l$ and for each $i \leq d$ a perfect pairing $$H^i(X, \mathbb{Q}_l) \otimes H^{2d-i}(X, \mathbb{Q}_l) \to H^{2d}(X, \mathbb{Q}_l) \simeq \mathbb{Q}_l$$ of finite-dimensional vector spaces induced by the cup-product. 

\textbf{(e) Lefshetz trace formula} Let $\varphi: X \to X$ be a regular map. Then  $$(\Gamma_{\varphi} \cdot \Delta)=\sum_i (-1)^iTr(\varphi^*, H^i(X, \mathbb{Q}_l)),$$ where $\varphi^*$ is the induced map in cohomology, $\Gamma_{\varphi}$ is the graph of $\varphi$, $\Delta$ is the diagonal of $X \times X$ and $(\Gamma_{\varphi} \cdot \Delta)$ is the intersection product (the number of fixed points counted with multiplicities). 

\textbf{ (f) Comparison theorem} If $X$ is a nonsingular variety over $\mathbb{C}$, then $$H^i(X, \mathbb{Q}_l) \otimes_{\mathbb{Q}_l} \mathbb{C} \simeq H^i(X_{\mathbb{C}}, \mathbb{C})$$ where the cohomology group on the right hand side is the ordinary cohomology group of $X_{\mathbb{C}}$. 

\textbf{ (g) Lifting to characteristic 0} Assume that $p \neq 0$ and $X$ can be lifted to a variety $X_1$ over a field $K$ of characteristic $0$. Let $K^{al}$ be the algebraic closure of $K$ and $X_{1, K^{al}}$ the corresponding variety over $K^{al}$. Then $$H^r(X, \mathbb{Q}_l) \simeq H^r(X_{1, K^{al}}, \mathbb{Q}_l).$$ In particular, the Betti numbers of $X$ equal those of $X_1$. 

Similar statements are true for the cohomology with compact support.

\end{thrm5}
The cup-product appearing in (c) and (d) can be (most easily) defined by using Cech cohomology as in the classical case.

Proving any of these results requires quite serious work. I will briefly mention the corresponding references and what one can expect to see there. I will also state an important consequence of (b) that we shall need to prove the Riemann hypothesis. 

\textbf{(a)} This is a consequence of the more general proper base-change theorem [2], Chapter 17 (which states roughly that: if the morphism of varieties $\pi: X \to S$ is proper and $\mathcal{F}$ is a constructible sheaf on $X$, then higher direct images $R^i \pi_* \mathcal{F}$ are constructible and we can compute the stalks of $R^i \pi_* \mathcal{F}$). The statement holds for all complete varieties (see [2], 19.1). 

\textbf{(b)}This is proved by first establishing a similar result for torsion sheaves. The proof of the first part is a straightforward induction (one needs to apply a spectral sequence along the way). See [2], Chapter 15. The proof of the second part is more complicated and requires a sophisticated induction from the case $d=1$ (proved separately) together with a limiting argument and the proper base-change theorem. 

We have the following consequence\footnote{Surjectivity (and that is all we will need) follows by considering the Gysin cohomology sequence of $X$ and $Z=Y \cap X$ (see [2], 28.3 for the proof and [2], Chapter 16 for the derivation of the Gysin sequence).} of the second statement in (b) (under the same assumptions):

\begin{thrm12}
For all $i \geq 2$ there are maps $$H^{i-2}(Y, \mathbb{Q}_l) \to H^i (X, \mathbb{Q}_l)$$ (cup-products with the cohomology class of a hyperplane) such that:

(i) For $i=d+1, \ H^{d-1}(Y, \mathbb{Q}_l)(-1) \to H^{d+1} (X, \mathbb{Q}_l)$ is surjective.

(ii) For $i > d+1, \ H^{i-2}(Y, \mathbb{Q}_l)(-1) \to H^i (X, \mathbb{Q}_l)$ is an isomorphism.
\end{thrm12}

For a thorough review of a generalization of this result to singular varieties see [26], Chapter 2. We have used the notion of the twisting sheaf that we are not going to discuss until Subsection 2.6. One can ignore this result until we refer to it in the proofs.

\textbf{(c)} This can be generalized in several ways. Firstly, the statement holds for all complete varieties. Secondly, exchanging $\mathbb{Q}_l$ by $\mathbb{Z}_l$ we find an analog of the topological Kunneth theorem for PIDs. Finally, the theorem can even be stated for arbitrary finite rings instead of the $p$-adics. The proof is most natural in the language of \textit{derived categories}. See [2], Chapter 22 for details. 

\textbf{(d)} This theorem is included here for completeness in a somewhat imprecise form (the isomorphism there is not canonical). We will discuss it in detail in Subsection 2.6. The theorem holds for any complete nonsingular variety $X$. We will also state a generalization to noncomplete varieties and locally constant $\mathbb{Q}_l$-sheaves. See Subsection 2.6 of the paper and 2.3 in [7] for a discussion and SGA 4, XVIII for complete proof.

\textbf{(e)} This theorem holds for any complete nonsingular variety $X$ (or for any incomplete variety such that $\varphi$ extends uniquely to the map on the completion of the variety and has only simple fixed points on the complement of the variety - see [2], Chapter 29). The proof mimics the proof of the topological statement (see [2], Chapter 25) but one needs to be familiar with the properties of the \textit{cycle map}. The theorem can also be generalized to $\mathbb{Q}_l$-sheaves as we shall see in the next section (the proof for that case is based on the notion of \textit{non-commutative traces} - see [2], Chapter 29). 

\textbf{(f)} The theorem holds for any nonsingular variety over $\mathbb{C}$ and cohomology groups with coefficients in a finite group. The proof is rather abstract but quite straightforward (see [2], Chapter 21). Using the notion of analytic varieties one can also generalize the theorem to singular varieties (and the proof is similar but longer - see SGA 4, XVI 4). 

\textbf{(g)} This holds for any variety $X$. We shall need this theorem to verify the "Betti numbers part" of the Weil conjectures. It is a consequence of the smooth base-change theorem (see [2], Chapter 20). 

I have to emphasize that neither the above results, though they are truly beautiful and allow one to prove the Weil conjectures, nor the mentioned generalizations constitute a full list of important results in $l$-adic (and etale) cohomology. 

The material covered in the next two subsections also appears in paragraphs 1 and 2 of [7] respectively.

\subsection{Cohomological interpretation of zeta and L-functions}

In this subsection we will provide a \textit{cohomological interpretation} of zeta functions (which proves most of the rationality conjecture) and move on to generalize this description to various $L$-functions.

Let $X_0$ be a nonsingular projective (or complete) variety over $\mathbb{F}_q$ and let $X$ be the corresponding variety over $\mathbb{\bar F}_q$. The Frobenius endomorphism introduced in Subsection 1.2 induces the map of cohomology spaces: $$F^*: H^i (X, \mathbb{Q}_l) \to H^i (X, \mathbb{Q}_l).$$

Grothendieck proved the following result:
\begin{lemma}
Under the assumptions of the above we have $$\#X^F=\sum_i (-1)^iTr(F^*, H^i(X, \mathbb{Q}_l))$$ (the right side, that is a priori an $l$-adic integer is an integer equal to the left side). 
\end{lemma}

\begin{proof}
For this to make sense we need to show that the fixed points of $X^F=X_0(\mathbb{F}_q)$ all have multiplicity $1$ in $(\Gamma_{\varphi} \cdot \Delta)$ - that is $(\Gamma_{\varphi} \cdot \Delta)_P=1$ for any fixed point $P$ of $F$. Then the statement follows by applying the Lefschetz trace formula. 

To show that $(\Gamma_{\varphi} \cdot \Delta)_P=1$ we calculate $(dF)_P$. Replacing $X_0$ with an affine neighborhood $U_0=Spec A_0, \ A_0=\mathbb{F}_q[x_1, \cdots, x_d]$, for every $i$ we have $x_i \circ F=x_i^q$ and so $$(dx_i)_P \circ (dF)_P=(dx_i^q)_P=qx_i^{q-1}(dx_i)_P=0.$$ Therefore, $(dF)_P=0$ and the next lemma implies that $(\Gamma_{\varphi} \cdot \Delta)_P=1$.

\end{proof}

\begin{lemma}
Let $\varphi:X \to X$ be a regular map and let $P \in X$ be a fixed point of $\varphi$. Then $(\Gamma_{\varphi} \cdot \Delta)_P=1$ if $1$ is not an eigenvalue of $(d \varphi)_P$. 
\end{lemma}

\begin{proof}
The proof is elementary and follows from the general criterion for $( Y \cdot Z)_P=1$ (where $Y, \ Z$ are closed subvarieties of $X$). See [2], 25.6 for details. 
\end{proof}

Observing that $F^n$ is the Frobenius endomorphism relative to $X_0(\mathbb{F}_{q^n})$ gives the following formula: $$\#X^{F^n}=\# X_0(\mathbb{F}_{q^n})=\sum_i (-1)^iTr(F^{*n}, H^i(X, \mathbb{Q}_l)) \ \ \textbf{(3)}.$$

\newpage

We need an elementary lemma: 
\begin{lemma}
If $F: V \to V$ is an endomorphism of a vector space over $k$, we have a formal power series identity: $$\log(\det(1-Ft, V)^{-1})=\sum_{n > 0} Tr(F^n, V)\frac{t^n}{n} \ \ \textbf{(4)}.$$
\end{lemma}

\begin{proof}
Factor $S(t)=\det(1-Ft, V)$ as $S(t)=\prod (1-c_it)$. Then (after possibly extending $k$) we may assume that there exists a basis such that the matrix of $\varphi$ is upper-triangular with $c_i$-s on the diagonal. Then the matrix of $\varphi^n$ is also upper-triangular with $c_i^n$-s on the diagonal and so we have $Tr(\varphi^n, V)=\sum c_i^n$. Then the statement follows by summing both sides of $$\log((1-c_i t)^{-1})=\sum_{n>0} c_i^n \frac{t^n}{n}$$ over $i$. 
\end{proof}

\begin{thrm6}
For any projective (or complete nonsingular) variety $X_0$ of dimension $d$ over $\mathbb{F}_q$ we have $$Z(X_0, t)=\frac{P_1(X_0, t)P_3(X_0, t) \cdots P_{2d-1}(t)}{P_0(X_0, t) P_2(X_0, t) \cdots P_{2d}(X_0, t)} \ \ \textbf{(5)},$$ where $$P_i(X_0, t)=\det(1-F^* t, H^i(X, \mathbb{Q}_l)).$$
\end{thrm6}

\begin{proof}
This results from the following computation: $$Z(X_0, t)=\exp \left( \sum_{n \geq 1}  \# X_0(\mathbb{F}_{q^n}) \frac{t^n}{n} \right) \overset{(3)} =\exp \left( \sum_{n \geq 1} \left(\sum_i (-1)^iTr(F^{*n}, H^i(X, \mathbb{Q}_l))\right)\frac{t^n}{n}\right)=$$ $$=\prod_i \left( \exp(\sum_{n \geq 1} Tr(F^{*n}, H^i(X, \mathbb{Q}_l)) \frac{t^n}{n}) \right)^{(-1)^{i}} \overset{(4)}= \ \prod_i P_i(X_0, t)^{(-1)^{i+1}},$$ where the third equality is obtained by moving the inner sum outside. 
\end{proof}

\begin{remark}
The right side of (5) is in $\mathbb{Q}_l(t)$ and the theorem implies that its Taylor expansion at $t=0$, a priori a formal series in $\mathbb{Q}_l [[t]]$ with constant coefficient one, is in $\mathbb{Z}[[t]]$ and is equal to the left side, also considered as a formal series in $t$. In particular, we have $Z(X_0, t) \in \mathbb{Z}[[t]] \cap \mathbb{Q}_l(t)$. Since $F^*$ acts as $1$ on $H^0(X, \mathbb{Q}_l)$ and as $q^d$ on $H^{2d}(X, \mathbb{Q}_l)$, we have $$P_0(X_0, t)=1-t, \ P_{2d}(X_0, t)=1-q^d t.$$
\end{remark}

\begin{remark}
This theorem is straightforwardly (one needs to replace cohomology groups by cohomology groups with compact support) generalized to an arbitrary variety\footnote{The study of Shimura varieties suggests that our definition of the zeta function is not in fact the correct one for noncomplete varieties $X_0$.} $X_0$ over $\mathbb{F}_q$ (possibly singular or noncomplete) and then one can similarly show that $Z( X_0, t) \in \mathbb{Q}(t)$. However, it is not known in general whether $P_i(X_0, t)$ are independent of $l$ nor even that they have coefficients in $\mathbb{Q}$. See [7] and [18], 289-298 for details. 
\end{remark}

Let me now consider locally constant $\mathbb{Q}_l$-sheaves and generalize this result to $L$-functions. We will follow the same path (but let $X_0$ be arbitrary for a change) though the first few steps are more subtle.

Let $X_0$ be an algebraic variety over $\mathbb{F}_q$, $\mathcal{F}_0$ a locally constant $\mathbb{Q}_l$-sheaf on $X_0$, $X$ the corresponding variety over $\mathbb{\bar F}_q$ and $\mathcal{F}$ the inverse image of $\mathcal{F}_0$ on $X$. We first need to define an endomorphism of $H_c^i(X, \mathcal{F})$. A similar map above (for the ordinary cohomology) was simply induced by $F$ and for $\mathcal{F}=\mathbb{Q}_l$ we can do the same here (since $F$ is finite, thus proper). However, if we try to do the same for an arbitrary $\mathcal{F}$, we will get a map $$\varphi: H_c^i(X, \mathcal{F}) \to H_c^i(X, F^*\mathcal{F}).$$ Luckily, there exists a canonical isomorphism $F^*: F^* \mathcal{F} \overset{\sim} \to \mathcal{F}$ (that is also a finite map) and we can define the required endomorphism as the composition\footnote{Sorry for the abuse of notation, I was just following the master.} $$F^*: H_c^i(X, \mathcal{F}) \to H_c^i(X, F^*\mathcal{F}) \to H_c^i(X, \mathcal{F}) \ \ \textbf{(6)}.$$

The isomorphism is constructed as follows (those unfamiliar with \textit{algebraic spaces} may skip this argument). First let $\mathcal{F}$ be the sheaf of sets (for the etale topology). Then we identify $\mathcal{F}_0$ with its etale space $f_0: [\mathcal{F}_0] \to X_0$\footnote{If $\mathcal{F}$ is a sheaf on $X$, then there always exists an algebraic space $[\mathcal{F}]$ equipped with an etale morphism $f: [\mathcal{F}] \to X$ such that $\mathcal{F}$ is the sheaf of local sections of $[\mathcal{F}]$.} and let $f: [\mathcal{F}] \to X$ be a similar etale space for $\mathcal{F}$ obtained by extension of scalars of $[\mathcal{F}_0]$. Then we have a commutative diagram 
\[
\begin{tikzcd}
{[\mathcal{F}]} \arrow{r}{F} \arrow[swap]{d}{f} & {[\mathcal{F}]} \arrow{d}{f} \\
X \arrow{r}{F} & X
\end{tikzcd}
\]
and a morphism $[\mathcal{F}] \to X \times_{(F, X, f)} [\mathcal{F}]=[F^* \mathcal{F}]$, that is an isomorphism because $f$ is etale. The inverse of this isomorphism defines the isomorphism $F^* \mathcal{F} \overset{\sim} \to \mathcal{F}$ that we seek. This construction is then generalized to $\mathbb{Q}_l$-sheaves. 

For $x \in |X|$ \footnote{The set of closed points of $X$.}, (6) defines a morphism that is  an endomorphism of $\mathcal{F}_x$ for $x \in X^F$.

The following formula is another result of Grothendieck (see [2], Chapter 29 for the proof in the case of curves): $$\sum_{x \in X^F} Tr(F_x^*, \mathcal{F}_x)=\sum_i (-1)^i Tr(F^*, H_c^i(X, \mathcal{F})).$$ 

Since the $n$-th iteration of $F^*$ defines morphisms $F_x^{*n}: \mathcal{F}_{F^n(x)} \to \mathcal{F}_x$ (that are endomorphisms for $x$ fixed under $F^n$) we have a similar formula for the iterations of $F^*$ $$\sum_{x \in X^{F^n}} Tr(F_x^{*n}, \mathcal{F}_x)=\sum_i (-1)^i Tr(F^{*n}, H_c^i(X, \mathcal{F})) \ \ \textbf{(7)}.$$

\begin{remark}
Let $x_0 \in |X|$, $Z$ the orbit corresponding to $F$ in $|X|$ and $x \in Z$. Then $Z$ has $\deg(x_0)$ elements. Denote by $F_{x_0}^*$ the endomorphism $F_x^{\deg(x_0)}$ of $\mathcal{F}_x$. $F^*$ is a local isomorphism so it induces an isomorphism of stalks and we may drop $x$ in the notation. Let $$\det(1-F_{x_0}^*t, \mathcal{F}_0)=\det(1-F_{x_0}^*t, \mathcal{F}_x).$$ We will use a similar notation for other functions of $(\mathcal{F}_x, F_{x_0}^*)$.
\end{remark}

We define the $L$-functions (or zeta functions of locally constant $\mathbb{Q}_l$-sheaves) as follows.

\begin{defn}
Let $X_0$ be a variety over $\mathbb{F}_q$ and $\mathcal{F}_0$ be a locally constant $\mathbb{Q}_l$-sheaf. Then we define $Z(X_0, \mathcal{F}_0, t)$ by $$\log Z(X_0, \mathcal{F}_0, t)=\sum_n \sum_{x \in X^{F^n}=X_0(\mathbb{F}_{q^n})} Tr(F_x^{*n}, \mathcal{F}_0)\frac{t^n}{n}.$$
\end{defn}
For the constant sheaf $\mathbb{Q}_l$ we recover $Z(X_0, t)$. One verifies as in the classical case that $$Z(X_0, \mathcal{F}_0, t)=\prod_{x \in |X_0|} \det(1-F_x^*t^{\deg(x)}, \mathcal{F}_0)^{-1}$$ and obtains (using (7) instead of (3) in the calculation) the cohomological interpretation of $L$-functions. 

\begin{thrm7}
For any variety $X_0$ of dimension $d$ over $\mathbb{F}_q$ and any locally-constant $\mathbb{Q}_l$-sheaf $\mathcal{F}_0$ on $X_0$ we have a formal series identity $$Z(X_0, \mathcal{F}_0, t)=\frac{P_1(X_0, t)P_3(X_0, t) \cdots P_{2d-1}(t)}{P_0(X_0, t) P_2(X_0, t) \cdots P_{2d}(X_0, t)}$$ where $$P_i(X_0, t)=\det(1-F^* t, H_c^i(X, \mathcal{F})).$$
\end{thrm7}

\begin{remark}
Up to this moment we have always used a geometric description for the action of Frobenius. However, sometimes it is useful to refer to the Galois side of the matter. One can verify that $Gal(\mathbb{\bar F}_q/\mathbb{F}_q)$ acts on $H_c^*(X, \mathcal{F})$ and that $F^*=\varphi^{-1}$ in $End(H_c^*(X, \mathcal{F}))$. Then we define the geometric Frobenius $F \in Gal(\mathbb{\bar F}_q/\mathbb{F}_q)$ as $F=\varphi^{-1}$ and obtain $F^*=F$. We will sometimes use the latter notation.
\end{remark}

\subsection{Poincare duality} 

This subsection serves three goals. First, we state the precise form of the Poincare duality theorem (and its generalization). Then, we define the \textit{Gysin map} that we will encounter a few times. Finally, we find some cohomology groups of a connected smooth curve and state yet another version of Poincare duality for the case of curves (these results will be used to complete the proof of the Riemann hypothesis).

Let $X$ be an algebraic variety \textit{purely} of dimension $n$ over an algebraically closed field $k$. We will now define the \textit{orientation sheaf} $\mathbb{Q}_l(n)$ of $X$ (and the twisting sheaves $\mathbb{Q}_l(i), \ i \in \mathbb{Z} \setminus \{0\}$). 

By analogy with the classical case of complex varieties (see [7], 2.1) we can say that "\textit{choosing an orientation}" of a variety $X$ over an algebraically closed field $k$ is the same as \textit{choosing an isomorphism of $\mathbb{Q}/\mathbb{Z}$ with the group of the roots of unity of $k$} (in the complex case this amounts to the choice of one of the two isomorphisms $x \to \exp(\pm 2 \pi i x)$ from $\mathbb{Q}/\mathbb{Z}$ to the group of the roots of unity of $\mathbb{C}$). If we are interested in $l$-adic cohomology ($l \neq char k$) of $X$ then we only need to consider the roots of order a power of $l$. Denote by $\mathbb{Z}/l^n(1)$ the group of the roots of unity of $k$ of order dividing $l^n$. Then the groups $\mathbb{Z}/l^n(1)$ form a \textit{projective system} with transition maps $$\sigma_{m, n}: \mathbb{Z}/l^m(1) \to \mathbb{Z}/l^n(1): x \to x^{l^{m-n}}$$ and we define $\mathbb{Z}_l(1)=\varprojlim \mathbb{Z}/l^n (1)$, $\mathbb{Q}_l(1)=\mathbb{Z}_l(1) \otimes_{\mathbb{Z}_l} \mathbb{Q}_l$, $\mathbb{Q}_l(i)=\overset{i} \otimes \mathbb{Q}_l(1)$ for $i > 0$ and $\mathbb{Q}_l(i)=\mathbb{Q}_l(-i \widecheck )$ (the \textit{dual sheaf}) for $i < 0$. 

Note that $\mathbb{Q}_l(1)$ is isomorphic to $\mathbb{Q}_l$ but the isomorphism is not canonical. 
The group of automorphisms of $k$ acts trivially on $\mathbb{Q}_l$ but not on $\mathbb{Q}_l(1)$. In particular, for $k=\mathbb{\bar F}_q$, the substitution of Frobenius $\varphi: x \to x^q$ acts by multiplication by $q$ on $\mathbb{Q}_l(1)$. 

We can now state the Poincare duality theorem: 
\begin{thrm8}
For $X$ nonsingular and complete (for example, projective), purely of dimension $n$, the bilinear form $$Tr(x \cup y): H^i(X, \mathbb{Q}_l) \otimes H^{2n-i}(X, \mathbb{Q}_l) \to \mathbb{Q}_l(-n)$$ is a perfect paring (it identifies $H^i(X, \mathbb{Q}_l)$ with the dual of $H^{2n-i}(X, \mathbb{Q}_l(n))$).
\end{thrm8}

One can generalize this theorem as follows (the case of curves will formally be enough for our purposes but it is still good to keep the higher-dimensional picture in mind). If $\mathcal{F}$ is a $\mathbb{Q}_l$-sheaf on an algebraic variety $X$ over an algebraically closed $k$, denote by $\mathcal{F}(r)$ the sheaf $\mathcal{F} \otimes \mathbb{Q}_l(r)$ (this sheaf is non-canonically isomorphic to $\mathcal{F}$).

\begin{thrm9}
Let $X$ be a nonsingular variety  purely of dimension $n$ and $\mathcal{F}$ be a locally constant sheaf. We denote by $\mathcal{\widecheck F}$ the dual of $\mathcal{F}$. The bilinear form $$Tr(x \cup y): H^i(X, \mathcal{F}) \otimes H_c^{2n-i}(X, \mathcal{\widecheck F}(n)) \to H_c^{2n}(X, \mathcal{F} \otimes \mathcal{ \widecheck F}(n)) \to H_c^{2n}(X, \mathbb{Q}_l(n)) \to \mathbb{Q}_l$$ is a perfect pairing. 
\end{thrm9}

We will now define the \textit{Gysin map}. Let $\pi: Y \to X$ be a proper map of nonsingular separated varieties over an algebraically closed field $k$. Let $m=\dim X$, $n=dim Y$ and $e=n-m$. Then there is a restriction map $$\pi^*: H_c^{2n-i}(X, \mathbb{Q}_l(n)) \to H_c^{2n-i} (Y, \mathbb{Q}_l(n))$$ and by duality we also get a map $$\pi_*: H^i(Y, \mathbb{Q}_l(e)) \to H^{i-2e}(X, \mathbb{Q}_l).$$ The latter is called the \textit{Gysin map}. Note that if $\pi$ is an endomorphism of $X$ (we shall use this for the Frobenius endomorphism $F$), then we get a map $\pi_*: H^i(X, \mathbb{Q}_l) \to H^i(X, \mathbb{Q}_l)$. The Gysin map satisfies a bunch of good properties summarized in [2], Chapter 24. We are going to use the fact that if $\pi$ above is a finite map of degree $\delta$ then $\pi_* \circ \pi^*=\delta$. 

The second form of Poincare duality allows us to prove:

\begin{fact}
Let $X$ be a connected smooth curve over an algebraically closed field $k$, $x$ a closed point of $X$ and $\mathcal{F}$ a locally constant $\mathbb{Q}_l$-sheaf. Then we have:

(i)$H_c^0(X, \mathcal{F})=0$ if $X$ is affine.

(ii)$H_c^{2}(X, \mathcal{F})=(\mathcal{F}_x)_{\pi_1(X, x)}(-1).$
\end{fact}

\begin{proof}
(i) simply states that $\mathcal{F}$ does not have sections with finite support. 

For the proof of (ii) we invoke the equivalence of Proposition 1 between the category of locally constant $\mathbb{Q}_l$ sheaves and that of (continious) $l$-adic representations of $\pi_1(X, x)$. Via this equivalence $H^0(X, \mathcal{F})$ is identified with the invariants of $\pi_1(X, x)$ acting on $\mathcal{F}_x$: $$H^0(X, \mathcal{F}) \overset{\sim} \to \mathcal{F}_x^{\pi_1(X, x)}.$$ According to the second form of Poincare duality, for $X$ nonsingular and connected of dimension $n$ we have $$H_c^{2n}(X, \mathcal{F})=H^0(X, \mathcal{\widecheck F}(n) \widecheck )=((\mathcal{\widecheck F}_x(n))^{\pi_1(X, x)} \widecheck ).$$ The duality exchanges \textit{invariants} (the largest invariant subspaces) to \textit{coinvariants} (the largest invariant quotients). So we have $$H_c^{2n}(X, \mathcal{F})=(\mathcal{F}_x)_{\pi_1(X, x)}(-n).$$ and (ii) is the case $n=1$. 
\end{proof}

Let $X$ be a connected smooth projective curve over an algebraically closed field $k$, $U$ an open in $X$, $j$ the inclusion $U\xhookrightarrow{} X$, $\mathcal{F}$ a locally constant $\mathbb{Q}_l$-sheaf on $U$ and $j_*\mathcal{F}$ the direct image of $\mathcal{F}$ (it is a constructible $\mathbb{Q}_l$-sheaf\footnote{This holds in general - the image of a constructible $\mathbb{Q}_l$-sheaf is always a constructible $\mathbb{Q}_l$-sheaf.}). We conclude this section with the following result.

\begin{thrm10}
The bilinear form $$Tr(x \cup y): H^i(X, j_*\mathcal{F}) \otimes H^{2-i}(X, j_*\mathcal{\widecheck F}(1)) \to H^{2}(X, j_*\mathcal{F} \otimes j_*\mathcal{ \widecheck F}(1)) \to$$ $$ \to H^{2}(X, j_*(\mathcal{F} \otimes \mathcal{ \widecheck F})(1)) \to H_c^{2}(X, \mathbb{Q}_l(1)) \to \mathbb{Q}_l$$ is a perfect pairing.
\end{thrm10}

\section{Proof of the Weil conjectures except for the Riemann hypothesis}

We will now prove all the conjectures (as stated in Subsection 1.1) except for the Riemann hypothesis (IV) and restate (IV). We will assume the Riemann hypothesis to show that $P_i(X_0, t)$ are integral.

\textbf{(I) Rationality} From Remark 2 (Subsection 2.5) we see that it suffices to prove that:

\textbf{(a)} $\mathbb{Z}[[t]] \cap \mathbb{Q}_l(t) \subset \mathbb{Q}(t)$

\textbf{(b)} $P_i(X_0, t)$ are integral polynomials. 

The proof of (a) is a good exercise on \textit{Hankel determinants} (see a more general statement in Bourbaki, Algebra, IV.5, Exercise 3). We sketch the argument here. Let $f(t) \in \mathbb{Z}[[t]] \cap \mathbb{Q}_l(t)$. Since $f(t)=\sum_n a_n t^n \in \mathbb{Z}[[t]]$ lies in $\mathbb{Q}_l(t)$, we see that its image in $\mathbb{Q}_l[[t]]$ is a Taylor expansion of a rational function. This implies that for all $M$ and $N$ large enough the Hankel determinants $$H_k=\det((a_{i+j+k})_{0 \leq i, j \leq M}) \ (k > N)$$ turn to zero. They turn to zero in $\mathbb{Q}_l$ if and only if they turn to zero in $\mathbb{Q}$ but this implies $f(t) \in \mathbb{Q}(t)$. 

We will now show (b) modulo the Riemann hypothesis. 

\begin{lemma}
Let $l$ be a prime number and $f(t)=\frac{g(t)}{h(t)}$ where $f(t) \in 1+t \cdot  \mathbb{Z}_l[[t]]$ and $g(t), h(t) \in 1 + t \cdot  \mathbb{Q}_l[[t]]$ are relatively prime. Then $g(t), h(t)$ have coefficients in $\mathbb{Z}_l$. 
\end{lemma}

\begin{proof}
It is enough to show that the coefficients of $g(t)$ and $h(t)$ have $l$-adic absolute values $\leq 1$. 

After taking a finite extension of $\mathbb{Q}_l$ we may assume that $h(t)$ splits. Let $h(t)=\prod (1-c_i t)$ and choose an arbitrary $i$. We need to show that $|c_i|_l \leq 1$. Indeed, suppose that $|c_i|_l > 1$ . Then we have $|c_i^{-1}|_l < 1$ and the power series $f(c_i^{-1})$ converges. But then $$f(t)h(t)=g(t) \Rightarrow f(c_i^{-1}) h(c_i^{-1})=g(c_i^{-1}).$$ This is a contradiction since $h(c_i^{-1})=0 \neq g(c_i^{-1})$. So we have $|c_i|_l \leq 1$ for all $i$ and $h(t) \in \mathbb{Z}_l[t]$. 

To show that $g(t) \in \mathbb{Z}_l[t]$ we apply the same argument to $f(t)^{-1} \in 1+t \cdot \mathbb{Z}_l[[t]]$. 
\end{proof}

\begin{lemma}
Let $Z(X_0, t)=\frac{P(t)}{Q(t)}$, with $P(t), Q(t) \in \mathbb{Q}[t]$ relatively prime (they exist by (a)). If they are chosen to have constant term $1$ then they have coefficients in $\mathbb{Z}$. 
\end{lemma}

\begin{proof}
According to the previous lemma the coefficients of $P(t)$ and $Q(t)$ will be $l$-adic integers for all prime $l$ (including $l=p$). This means that they are integers. 
\end{proof}

Let $P_i (X_0, t)=\det (1-F^*t, H^i (X, \mathbb{Q}_l))$. Then the Riemann hypothesis implies that all the $P_i(X_0, t)$ are relatively prime. Hence the fraction in Lemma 4 is irreducible and we have $$ P(t)=\prod_{i \ odd} P_i(X_0, t), \ Q(t)=\prod_{i \ even} P_i(X_0, t).$$ 
Therefore, inverse roots of $P_i(X_0, t)$ are algebraic integers which implies that $P_i(X_0, t)$ has integer coefficients.
\begin{remark}
Let $K$ be the subfield of the algebraic closure $\mathbb{\bar Q}_l$ of $\mathbb{Q}_l$ generated over $\mathbb{Q}$ by the roots of $R(t)=P(t)Q(t)$. By the Riemann hypothesis, the roots of $P_i(t)$ are the roots of $R(t)$ such that all their complex conjugates have absolute value $q^{-\frac{i}{2}}$. This description of the roots of $P_i(X_0, t)$ is independent of $l$, so $P_i(X_0, t)$ are also independent of $l$. 
\end{remark}

\textbf{(II) Functional equation}
Recall that $d$ is the dimension of $X$ and consider the duality pairing $$H^{2d-i}(X, \mathbb{Q}_l) \otimes H^i (X, \mathbb{Q}_l(d)) \to H^{2d}(X, \mathbb{Q}_l) \overset{\eta_X} \to \mathbb{Q}_l.$$
The Frobenius endomorphism $F: X \to X$ induces the maps $F^*$ and the Gysin maps $F_*$ (see Subsection 2.6), all of which are endomorphisms. Then the following holds by the definition of $F_*$ (and is equivalent to it): $$\eta_X(F_* (x) \cup x')=\eta_X(x \cup F^*(x')), \ x \in H^{2d-i}(X, \mathbb{Q}_l), \ x' \in H^i (X, \mathbb{Q}_l(d)).$$ Therefore, we see that the eigenvalues of $F^*$ acting on $H^i (X, \mathbb{Q}_l(d))$ are the same as the eigenvalues of $F_*$ acting on $H^{2d-i}(X, \mathbb{Q}_l)$. But $F$ is finite of deqree $q^d$ so we have $F^*=q^d/F_*$. This implies that if $\alpha_1, \cdots, \alpha_s$ are the eigenvalues of $F^*$ acting on $H^i(X, \mathbb{Q}_l)$ then $q^d/\alpha_1, \cdots, q^d/\alpha_s$ are the eigenvalues of $F^*$ acting on $H^{2d-i}(X, \mathbb{Q}_l)$. 

The rest of the proof is a direct calculation. We have (in the Galois language): $$\det(1-F q^{-d} t^{-1} , H^i(X, \mathbb{Q}_l))=(q^d t)^{-\beta_i} \det(F, H^i(X, \mathbb{Q}_l)) \cdot (-1)^{\beta_i} \det(1-Ft, H^{2d-i}(X, 
\mathbb{Q}_l)).$$ (where $\beta_i$ is the dimension of the $i$-th cohomology group). 

Suppose that $i \neq d$. Then $$\det(F, H^i(X, \mathbb{Q}_l)) \cdot \det(F, H^{2d-i}(X, \mathbb{Q}_l))=q^{\beta_i d}.$$

Consider the case $i=d$ and let $N_{+}$ (resp. $N_{-}$) be the number of the eigenvalues of $F$ acting on $H^d(X, \mathbb{Q}_l)$ equal to $q^{\frac{d}{2}}$ (resp. $-q^{\frac{d}{2}}$). The remaining eigenvalues form pairs $\alpha \neq q^d/\alpha$ so $\beta_d-N_{+}-N_{-}$ is even. Then we have $$\det(F, H^d(X, \mathbb{Q}_l))=q^{d(\beta_q-N_{+}-N_{-})/2}q^{d(N_{+}+N_{-})/2}(-1)^{N_{-}}=q^{d\beta_d/2}(-1)^{N_{-}},$$ where $d\beta_d$ is always even since the Poincare duality pairing is skew-symmetric on $H^d(X, \mathbb{Q}_l)$ for odd $d$. 

We use the cohomological interpretation of the zeta function to obtain the functional equation:
$$Z(X_0, q^{-d}t^{-1})=\prod_{i=0}^{2d} \det(1-F q^{-d}t^{-1}, H^i(X, \mathbb{Q}_l))^{(-1)^{i+1}}=$$ $$=(q^d t)^{\chi} q^{\frac{-d \chi }{2}} (-1)^{N_{+}}\prod_{i=0}^{2d} \det(1-F t, H^{2d-i}(X, \mathbb{Q}_l))^{(-1)^{i+1}}=(-1)^{N_{+}} q^{\frac{d \chi}{2}} t^{\chi} Z(X_0, t),$$ where $\chi=\sum_i (-1)^i \beta_i$. 

\textbf{(III) Betti numbers} 

This immediately follows from the Lefschetz trace formula (e) (for $\varphi=Id$) and the Lifting to characteristic 0 (g) properties of $l$-adic cohomology (see Subsection 2.4). 

\textbf{(IV) Riemann hypothesis} 

In view of the cohomological interpretation for the zeta function we can reformulate the Riemann hypothesis. Sections 4-8 are devoted to the proof of the following theorem (and are similar to parapraphs 3-7 of [7]).
\begin{thrm11}
Let $X_0$ be a nonsingular projective variety over the finite field $\mathbb{F}_q$ of characteristic $p$. For each $i$ and each $l \neq p$ the eigenvalues of $F^*$ acting on $H^i(X, \mathbb{Q}_l)$ are algebraic numbers of absolute value $|\alpha|=q^{\frac{i}{2}}$.
\end{thrm11}

I am going to use the shorthands $DT$ for the whole Deligne's theorem, $DT(i)$ for its restriction to the eigenvalues of $F^*$ acting on a single cohomology group $H^i(X, \mathbb{Q}_l)$ and $DT(V)$ for the restriction to the eigenvalues of $F^*$ acting on a subspace $V \subset H^i (X, \mathbb{Q}_l)$ stable under $F^*$.

\begin{remark}
Actually, we can say that $\alpha$-s are algebraic numbers all of which complex conjugates (roots of $P(t) \in \mathbb{Q}[t]$ in $\mathbb{C}$, where $P(t)$ is a monic and irreducible polynomial such that $P(\alpha)=0$) have absolute value $q^{\frac{i}{2}}$. It is actually quite unusual for an algebraic number to have all the complex conjugates with the same absolute value (for instance, $1+\sqrt{2}$ does not have this property). 
\end{remark}

I would like to conclude this section by informally outlining the steps of the proof.

In view of the Kunneth isomorphism (property (c) of the $l$-adic cohomology of Subsection 2.4) and the fact that this isomorphism is compatible with the action of $F^*$ one can argue that if the theorem is true for $X_0$ and $Y_0$ then it should also be true for the product $X_0 \times Y_0$ and if every variety can be represented as a product of curves then we can argue by dimension and call it a day. 

Unfortunately (or, depending on your perspective, very fortunately), this is not the case but Deligne takes a cue from this and (once the reductions of Section 4 are made) argues by induction (though in the end we will not have to refer to the known result for curves). The induction works by taking a map $f: X \to \mathbb{P}^1$ for $X$ (rather for $\tilde X$ such that it suffices to prove the theorem for $\tilde X$) that looks similar to the first projection in the case of $X \simeq \mathbb{P}_1 \times Y$ and fails (in a certain sense) only at a finite number of points (this is where Lefschetz theory of Section 6 comes of use). By the use of the Leray spectral sequence (see Subsection 2.3) one relates the cohomology of $\mathbb{P}_1$ with that of $X$ and thus makes a reduction to the last irreducible case where the ideas acquired by Deligne while reading the paper of Rankin (see Section 5) allow him to conclude the proof. 

It might be useful to reread the above paragraph after studying Sections 4-7 and before the line of reasoning described above is implemented in Section 8.

\section{Preliminary reductions}

I will first make a few obvious remarks. 

Let $X_0, Y_0$ be nonsingular projective varieties over $\mathbb{F}_q$ and $X, Y$ be the corresponding varieties over $\mathbb{\bar F}_q$. It is easy to see that:

(i) If $H^i(Y, \mathbb{Q}_l) \to H^i(X, \mathbb{Q}_l)$ is a surjection, then $DT(i)$ for $Y$ $\Rightarrow$ $DT(i)$ for $X$. 

(ii) If $H^i(X, \mathbb{Q}_l) \xhookrightarrow{} H^i(Y, \mathbb{Q}_l)$ is an embedding, then $DT(i)$ for $Y$ $\Rightarrow$ $DT(i)$ for $X$. 

More generally, one can prove the following statement:
\begin{fact}
Given two subspaces $V \subset H^i (X, \mathbb{Q}_l)$ and $W \subset H^j(Y, \mathbb{Q}_l)$ stable under $F^*$ and a $\mathbb{Q}_l$-linear map $\varphi: V \to W$ such that for all $v \in V$ we have $$\varphi(F^* v)=q^{(i-j)/2}F^*\varphi(v)$$ (that is our map is compatible with $F^*$), the following hold:

(i) If $\varphi$ is surjective, then $DT(V)$ for $X$ $\Rightarrow$ $DT(W)$ for $Y$. 

(ii) If $\varphi$ is injective, then $DT(W)$ for $Y$ $\Rightarrow$ $DT(V)$ for $X$.

\end{fact}

We shall also need the following:
\begin{fact}
It suffices to prove Deligne's theorem for $\mathbb{F}_q$ replaced by a finite extension $\mathbb{F}_{q^n}$. 
\end{fact}

\begin{proof}
The Frobenius map $F_n: X \to X$ relative to $\mathbb{F}_{q^n}$ is $F^n$ so if $\alpha$ is an eigenvalue of $F$ on $H^i (X, \mathbb{Q}_l)$ then $\alpha^n$ is an eigenvalue of $F_n$ on $H^i (X, \mathbb{Q}_l)$. Therefore, $\alpha^n$ satisfies the condition of $DT$ relative to $q^n$ if and only if $\alpha$ satisfies the condition of $DT$ relative to $q$. 
\end{proof}

We will sometimes use the results above without reference. Let's turn to making the reductions. 
\begin{reduction1}
$DT(d)$ for all absolutely (geometrically) irreducible nonsingular projective varieties of dimension $d$ implies $DT$. 
\end{reduction1}
That is, we only need to worry about the middle-dimensional cohomology groups. 
\begin{proof}
We will need the following results:

\textbf{(a)} If $X_0$ is purely of dimension $d$, $DT(i)$ for $X_0$ is equivalent to $DT(2d-i)$ for $X_0$.

This follows from the fact that if $\alpha$ is an eigenvalue of $F^*$ acting on $H^i(X, \mathbb{Q}_l)$ then $q^d/\alpha$ is an eigenvalue of $F^*$ acting on $H^{2d-i}(X, \mathbb{Q}_l)$. We have shown that this follows from Poincare duality in our proof of the Functional equation conjecture (see Section 3). 

\textbf{(b)} If $X_0$ is a union of irreducible $\{ X_0^{\alpha} \}_{\alpha \in I}$, $DT$ for $X_0$ is equivalent to the collection of $DT$ for all $X_0^{\alpha}, \ \alpha \in I$.

This is obvious since the cohomology groups split as direct products. 

\textbf{(c)} If $X_0$ is purely of dimension $d$, $Y_0$ is a smooth hyperplane section of $X_0$ and $i < d$, then $W(Y_0, i) \Rightarrow W(X_0, i)$.

This follows from the weak Lefschetz theorem (see Subsection 2.4) together with the fact that the map $H^{i-2}(Y, \mathbb{Q}_l)(-1) \to H^i (X, \mathbb{Q}_l)$ is compatible with $F^*$.

Reduce the general case as follows (we can assume that $d \geq 2$): 

-by (b) we can assume that $X_0$ is purely of dimension $d$.

-by (a) we can also assume that $i \geq d$.

-by (c), Proposition 5 and Bertini's hyperplane theorem\footnote{It states the existence of $Y_0$ in (c).} (Harthorne, II.8.18) we can further assume that $i=d$.

-by (b) and Proposition 5 we can finally assume that $X_0$ is absolutely irreducible. 

This concludes the proof. 
\end{proof}

\begin{remark}
As noted by A.Mellit, it is possible to avoid using the weak Lefschetz theorem here by applying a Kunneth formula argument. We know how the Frobenius endomorphism acts on $H^0(X, \mathbb{Q}_l)$ and $H^{2d}(X, \mathbb{Q}_l)$. Suppose that $d \leq i$. Then if we tensor $H^i(X, \mathbb{Q}_l)$ by itself $d$ times and $H^0(X, \mathbb{Q}_l)$ by itself $i-d$ times then the whole product lands in $H^{id}(X^{id}, \mathbb{Q}_l)$ (the middle-dimensional cohomology group).
Similarly for $d >i$.
\end{remark}

Consider the following statement.

\begin{dt}
Let $X_0$ be a nonsingular absolutely irreducible projective variety of even dimension $d$ over $\mathbb{F}_q$, $X$ the corresponding variety over $\mathbb{\bar F}_q$ and $\alpha$ an eigenvalue of $F^*$ acting on $H^d(X, \mathbb{Q}_l)$. Then $\alpha$ is an algebraic number all of which complex conjugates, still denoted $\alpha$, satisfy $$q^{\frac{d}{2}-\frac{1}{2}} \leq |\alpha| \leq q^{\frac{d}{2}+\frac{1}{2}}.$$ 
\end{dt}

\begin{reduction2}
$DR$ implies $DT$.
\end{reduction2}

\begin{proof}
By the first reduction it suffices to show that $DR$ implies $DT(d)$ for absolutely irreducible varieties. Let $X_0$ be a nonsingular absolutely irreducible projective variety of dimension $d$ (not necessarily even), $\alpha$ be an eigenvalue of $F^*$ on $H^d(X, \mathbb{Q}_l)$ and $k$ be an even integer.  Then Kunneth isomorphism implies that $\alpha^k$ is an eigenvalue of $F^*$ acting on $H^{kd}(X^k, \mathbb{Q}_l)$. By $DR$ we have:  $$q^{\frac{kd}{2}-\frac{1}{2}} \leq |\alpha^k| \leq q^{\frac{kd}{2}+\frac{1}{2}}$$ and $$q^{\frac{d}{2}-\frac{1}{2k}} \leq |\alpha| \leq q^{\frac{d}{2}+\frac{1}{2k}}.$$ Letting $k$ go to infinity, we find that $$|\alpha|=q^{\frac{d}{2}}.$$  
\end{proof}

We are going to prove $DR$ in Section 8 and thus conclude the proof of $DT$.

\section{The Main Lemma}

In this section we discuss the technical heart of the proof. These are the ideas commonly referred to as the Main Lemma in English literature (french. la majoration fondamentale). Deligne was able to deduce the results of this section after reading the lectures of Rankin [20] while studying automorphic forms. The main idea of Rankin (using the result of Landau) was that when one has a product defining the zeta function, one can get information on the \textit{local factors} (we will soon define them in Lemma 6) from the information on the pole of the zeta function itself. It occurred to Deligne to try a similar approach in a different situation, which eventually lead him to the complete proof of the Riemann hypothesis. 

We will now state the main result of this section. Let $U_0$ be a curve over $\mathbb{F}_q$, $U$ the corresponding curve over $\mathbb{\bar F}_q$, $u$ a closed point of $U$, $\mathcal{F}_0$ a locally constant $\mathbb{Q}_l$-sheaf on $U_0$ and $\mathcal{F}$ its inverse image on $U$. 

Let $\beta \in \mathbb{Z}$ \footnote{One may want to define weights for $\beta \in \mathbb{Q}$ but we will not need this.}. We will say that $\mathcal{F}_0$ has \textit{weight} $\beta$ if for all $x \in |U_0|$ (the closed points of $U$), the eigenvalues of $F^*_x$ acting on $\mathcal{F}_0$ (see Remark 4 of Subsection 2ю5) are algebraic numbers all of which complex conjugates have absolute value $N(x)^{\beta/2}=q^{\deg(x)\beta/2}$. It is easy to see that if $\mathcal{F}_0$ has weight $\beta$ then $\overset{r} \otimes \mathcal{F}_0$ has weight $r\beta$. In particular, one can show that $\mathbb{Q}_l(1)$ has weight $-2$ and $\mathbb{Q}_l(i)$ has weight $-2i$ (see [2], 30.5).

\begin{thrm13}
Assume that:

\textbf{ (a) Bilinear form} $\mathcal{F}_0$ is equipped with a bilinear skew-symmetric nondegenerate form $$\psi: \mathcal{F}_0 \otimes \mathcal{F}_0 \to \mathbb{Q}_l(-\beta).$$

\textbf{ (b) Big geometric monodromy} The image of $\pi_1(U, u)$ in $GL(\mathcal{F}_u)$ is an open subgroup of the symplectic group\footnote{For a vector space $V$ and a skew-symmetric bilinear form $\psi$ the symplectic group is defined as $$Sp(V, \psi)=\{ \alpha \in GL(V) \ \psi(\alpha v, \alpha v')=\psi(v, v'), \ v, v' \in V \}.$$} $Sp(\mathcal{F}_u, \psi_u)$.

\textbf{ (c) Rationality} For all $x \in |U_0|$, the polynomial $\det(1-F_x t, \mathcal{F}_0)$ has rational coefficients. 

Then $\mathcal{F}$ has weight $\beta$.

\end{thrm13}

\begin{remark}
There is a different way to formulate condition (c) of the above theorem. Let $K$ be a field such that $\mathbb{Q} \subset K$ and let $\varphi: V \to V$ be an endomorphism of a vector space $V$ over $\mathbb{K}$. Then $\det(1-\varphi t, V)$ has rational coefficients if and only if there exists a basis for $V$ such that the matrix of $\varphi$ has rational coefficients (one direction is obvious and the converse follows from the theory of rational forms for matrices). We will say that an endomorphism $\varphi$ is rational if it satisfies these equivalent conditions. 
\end{remark}

\begin{proof}

We may and do assume that $U$ is affine and that $\mathcal{F} \neq 0$. The following two lemmas are straightforward. 

\begin{lemma}
Let $2k$ be an even integer and $x \in |U_0|$. The local factors $\det(1-F_x t^{\deg(x)}, \overset{2k} \otimes \mathcal{F}_0)^{-1}$ are formal power series with positive rational coefficients.
\end{lemma}

\begin{proof}
The rationality assumption (c) ensures that for all $n$ we have $Tr(F_x^n, \mathcal{F}_0) \in \mathbb{Q}$. For any endomorphism $F:V \to V$ of a finite-dimensional vector space over a field $k$ and for all $i$ one has $$Tr(F, \overset{i} \otimes V)=Tr(F, V)^i.$$ Hence $$Tr(F_x^n, \overset{2k} \otimes \mathcal{F}_0)=Tr(F_x^n, \mathcal{F}_0)^{2k}$$ is a positive rational. We have established the following formula in Lemma 3 (Subsection 2.5) $$\log(\det((1-Ft, V)^{-1})=\sum_{n > 0} Tr(F^n, V)\frac{t^n}{n}.$$
So $\log(\det(1-F_x t^{\deg(x)}, \overset{2k} \otimes \mathcal{F}_0)^{-1})$ has positive rational coefficients and the coefficients of its exponentiation are therefore also positive.

\end{proof}

\begin{lemma}
Let $\{ f_i=\sum_n a_{i, n}t^n \}_{i \geq 1}$ be a sequence of formal power series with positive real coefficients. Assume that the order of $f_i-1$ tends to infinity with $i$ and let $f=\prod_i f_i$. Then for any $i$ the radius of convergence of $f_i$ is greater or equal than that of $f$. 

If $f$ and all the $f_i$ are Taylor expansions of meromorphic functions, then $$inf \{ |z| \ | f(z)=\infty \} \leq inf \{ |z| \ | f_i(z)=\infty \}.$$
\end{lemma}

\begin{proof}
Let $f=\sum_n a_n t^n$. Then we have $a_{i, n} \leq a_n$ and this implies the first part.
For the second part simply note that those numbers are the radii of convergence. 
\end{proof}

Assumption (b) of the theorem implies the following lemma.
\begin{lemma}
$\pi_1(U, u)$ is Zariski dense in $Sp(\mathcal{F}_u, \psi_u)$ and therefore $$Hom(\overset{2k} \otimes \mathcal{F}_u, \mathbb{Q}_l)^{\pi_1(U, u)}=Hom(\overset{2k} \otimes \mathcal{F}_u, \mathbb{Q}_l)^{Sp(\mathcal{F}_u, \psi_u)}.$$
\end{lemma}

\begin{proof}
$\pi_1(U, u)$ is a Galois group so it is compact and its image $\bar \pi$ in $Sp(\mathcal{F}_u, \psi_u)$ is closed. By the $l$-adic version of Cartan's theorem (see [21], LG 5.42) $\bar \pi$ is a Lie subgroup of $Sp$. Since it is open, it has the same dimension as a Lie group. Let $G \subset Sp$ be the Zariski closure of $\bar \pi$ in $Sp$ (the smallest closed subvariety of $Sp$ defied over $\mathbb{Q}_l$ such that $\bar \pi \subset G(\mathbb{Q}_l)$). Then $G$ is an algebraic subgroup of $Sp$ and we want to show that $G=Sp$. Consider $$T_e \bar \pi \subset T_e G \subset T_e Sp.$$ Since $T_e \bar \pi= T_e Sp$, $\dim G=\dim Sp$ and so $G=Sp$ since $Sp$ is connected. 

Let $f: \overset{2k} \otimes \mathcal{F}_u \to \mathbb{Q}_l$ be a map fixed by $\pi_1(U, u)$. For $\tau \in Sp (\mathcal{F}_u, \psi_u)$ fixing $f$ is an algebraic condition on $\tau$ so if $f$ is fixed by $\bar \pi$ it is also fixed by its Zariski closure. 
\end{proof}

Obviously, $$Hom(\overset{2k} \otimes \mathcal{F}_u, \mathbb{Q}_l)^{Sp}=Hom((\overset{2k} \otimes \mathcal{F}_u)_{Sp}, \mathbb{Q}_l).$$ If $\psi_1, \cdots, \psi_N$ is a basis for $Hom(\overset{2k} \otimes \mathcal{F}_u, \mathbb{Q}_l)^{Sp}$ then the map $$ \overset{2k} \otimes \mathcal{F}_u \to \mathbb{Q}_l^N: x \to \prod_{i=1}^N \psi_i (x)$$ induces an isomorphism $$(\overset{2k} \otimes \mathcal{F}_u)_{Sp} \overset{\sim} \to \mathbb{Q}_l^{N}.$$
We will now use invariant theory to choose a basis $\{ \psi_i \}$ such that the action of $Gal(\mathbb{\bar F}_q/\mathbb{F}_q)$ transfers to an action on $\mathbb{Q}_l^N$. For any partition $P$ of $\{1, \cdots, 2k\}$ into $k$ two element sets $\{ i_{\alpha}, j_{\alpha} \} \ (i_{\alpha} < j_{\alpha})$ we have an invariant form $$\psi_P: \overset{2k} \otimes \mathcal{F}_u \to \mathbb{Q}_l(-k\beta): x_1 \otimes \cdots \otimes x_{2k} \to \prod_{\alpha} \psi_u (x_{i_\alpha}, x_{j_{\alpha}}).$$ 
\begin{lemma}
The invariant forms $\psi_P$ span $Hom(\overset{2k} \otimes \mathcal{F}_u, \mathbb{Q}_l)^{Sp}$.
\end{lemma}

\begin{proof}
See [22], Chapter VI, par. 1 or [23], Appendix F for this invariant theory result. 
\end{proof}

Therefore, a basis of the forms of type $\psi_P$ induces an isomorphism $$(\overset{2k} \otimes \mathcal{F}_u)_{\pi_1}=(\overset{2k} \otimes \mathcal{F}_u)_{Sp} \overset{\sim} \to \mathbb{Q}_l(-k\beta)^{N}.$$
Then, according to Proposition 3 (Subsection 2.6) we have  $$H_c^2(U, \overset{2k} \otimes \mathcal{F}) \simeq \mathbb{Q}_l(-k\beta-1)^N$$ and  $$H_c^0(U, \overset{2k} \otimes \mathcal{F})=0,$$ so the cohomological interpretation of $L$-functions gives $$Z(U_0, \overset{2k} \otimes \mathcal{F}_0, t)=\frac{\det(1-F^*t, H^1(U, \overset{2k} \otimes \mathcal{F}))}{(1-q^{k\beta+1}t)^N}.$$ We shall only need the fact that  the poles of this zeta function have absolute value $t=1/q^{k\beta+1}$. If $\alpha$ is an eigenvalue of $F_x$ on $\mathcal{F}_0$, then $\alpha^{2k}$ is an eigenvalue of $F_x$ on $\overset{2k} \otimes \mathcal{F}_0$. Now, let $\alpha$ be any complex conjugate of the original $\alpha$. Then $1/\alpha^{2k/\deg(x)}$ is a pole of $\det(1-F_x t^{\deg(x)}, \overset{2k} \otimes \mathcal{F}_0)^{-1}$. By Lemmas 6 and 7 we have $$|1/q^{k\beta+1}| \leq |1/\alpha^{2k/\deg(x)}|$$ or $$|\alpha| \leq q^{\deg(x)(\frac{\beta}{2}+\frac{1}{2k})}$$ and letting $k$ go to infinity gives $$|\alpha| \leq q^{\frac{\deg(x)\beta}{2}}.$$ The existence of the pairing $\psi$ implies that $q^{\deg(x)\beta} \alpha^{-1}$ is an eigenvalue, so we have the inequality $$|q^{\deg(x)\beta} \alpha^{-1}| \leq q^{\frac{\deg(x)\beta}{2}}$$ or $$q^{\frac{\deg(x)\beta}{2}} \leq |\alpha|.$$ This completes the proof.

\end{proof}

\newpage

\begin{cor}
Let $\alpha$ be an eigenvalue of $F^*$ acting on $H_c^1(U, \mathcal{F})$. Then $\alpha$ is an algebraic number all of which complex conjugates satisfy $$|\alpha| \leq q^{\frac{\beta+1}{2}+\frac{1}{2}}$$ and the action of $F^*$ on $H_c^1(U, \mathcal{F})$ is rational.

\end{cor}

\begin{proof}
Since $U$ is affine Proposition 3 (Subsection 2.6) gives $$H_c(U, \mathcal{F})=0$$ and $$H_c^2 (U, \mathcal{F})=(\mathcal{F}_u)_{\pi_1(U, u)}(-1).$$ There are no $Sp$-invariant linear forms $\mathcal{F}_u \to \mathbb{Q}_l$ (or $\overset{n} \otimes \mathcal{F}_u \to \mathbb{Q}_l$ for any odd $n$, by invariant theory) so assumption (b) of the theorem implies that this cohomology group is also zero. The cohomological interpretation of $L$-functions gives $$Z(U_0, \mathcal{F}_0, t)=\det(1-F^*t, H_c^1(U, \mathcal{F})).$$ 

The left hand side is a formal series with rational coefficients, given its representation as a product and the rationality assumption. The right hand side is therefore a polynomial with rational coefficients and $1/\alpha$ is its root, so the action of $F^*$ is rational and $\alpha$ is algebraic. To complete the proof it suffices to show that the infinite product $$Z(U_0, \mathcal{F}_0, t)=\prod_{x \in |U_0|} \det(1-F_x^*t^{\deg(x)}, \mathcal{F}_0)^{-1}$$ converges absolutely for $|t| <  q^{-\frac{\beta}{2}-1}$. 
Let $N$ be the rank of $\mathcal{F}$ and let $\alpha_{i,x}$ be the eigenvalues of $F^*_x$ acting on $\mathcal{F}_x$. Then $$\det(1-F^*_xt^{\deg(x)}, \mathcal{F}_0)=\prod_{i=1}^N (1-\alpha_{i, x}t)$$ and the Main Lemma implies that $|\alpha_{i, x}|=q^{\deg(x)\beta/2}$. The convergence of the infinite product is equivalent to that of the series $$\sum_{i, x} |\alpha_{i, x} t^{\deg(x)}|.$$ 

For $|t|=q^{-\frac{\beta}{2}-1-\varepsilon}\ (\varepsilon >0)$ we have $$\sum_{i, x} |\alpha_{i, x} t^{\deg(x)}|=N\sum_x q^{\deg(x)(-1-\varepsilon)}$$ so it suffices to prove that $\sum_x q^{\deg(x)(-1-\varepsilon)}$ converges. Every closed point of $U_0$ of degree $n$ contributes at least $1$ (in fact, exactly $n$) elements to $U_0(\mathbb{F}_{q^n})$. Since $U_0 \subset \mathbb{A}^1$, the number of closed points of degree $n$ is bounded above by $q^n$ and we have $$\sum_x q^{\deg(x)(-1-\varepsilon)} \leq \sum_n q^nq^{n(-1-\varepsilon)}=\sum_n q^{-n\varepsilon} < \infty,$$ so the corollary is proved.

Note that we have only used the facts that $\mathcal{F}_0$ has weight $\beta$ and $(\mathcal{F}_u)_{\pi_1(U, u)}(-1)=0$. 
\end{proof}

\begin{cor}
Let $j_0$ be the inclusion of $U_0$ in $\mathbb{P}_{\mathbb{F}_q}^1$, $j$ that of $U$ into $\mathbb{P}^1$ and $\alpha$ an eigenvalue of $F^*$ acting on $H^1(\mathbb{P}^1, j_*\mathcal{F})$. Then $\alpha$ is an algebraic number all of which complex conjugates satisfy $$q^{\frac{\beta+1}{2}-\frac{1}{2}} \leq |\alpha| \leq q^{\frac{\beta+1}{2}+\frac{1}{2}}$$ and the action of $F^*$ on $H^1(\mathbb{P}^1, j_*\mathcal{F})$ is rational.
\end{cor}

\begin{proof}
From the long exact sequence in cohomology defined by the short exact sequence\footnote{One may also use the exact sequence $$0 \to j_! \mathcal{F} \to j_*\mathcal{F} \to i_*i^*j_* \mathcal{F} \to 0$$ (for $i: Z \to \mathbb{P}^1$ the inclusion of the complement $Z=\mathbb{P}^1 \setminus U$ into $\mathbb{P}^1$) derived from that in [2], 8.15. by applying $j_*$.} $$0 \to j_! \mathcal{F} \to j_*\mathcal{F} \to j_*\mathcal{F}/j_! \mathcal{F} \to 0$$ ($j_!$ is the extension by 0) we obtain a surjection $$H_c^1(U, \mathcal{F}) \to H^1( \mathbb{P}, j_* \mathcal{F}),$$ so Corollary 1 implies that the action of $F^*$ on $H^1(\mathbb{P}^1, j_*\mathcal{F})$ is rational and we have: $$|\alpha| \leq q^{\frac{\beta+1}{2}+\frac{1}{2}}.$$ 
The sheaf $\mathcal{\widecheck F}(1)$ satisfies the same assumptions as $\mathcal{F}$ with $\beta$ replaced by $-2-\beta$. Hence the action of $F^*$ on $H^{2-i}(\mathbb{P}^1, j_*\mathcal{\widecheck F}(1))$ is rational and its eigenvalues $\alpha'$ satisfy $$|\alpha'| \leq q^{\frac{-\beta-1}{2}+\frac{1}{2}}.$$

From the third form of Poincare duality we have a canonical nondegenerate pairing $$H^i(X, j_*\mathcal{F}) \otimes H^{2-i}(X, j_*\mathcal{\widecheck F}(1)) \to H_c^{2}(X, \mathbb{Q}_l(1)) \overset{\sim} \to \mathbb{Q}_l$$ so each eigenvalue $\alpha$ is an inverse of an $\alpha'$ and we have $$|\alpha| \geq q^{\frac{\beta+1}{2}-\frac{1}{2}}.$$
\end{proof}

The results of this section will play a crucial role in Section 8. They suffice to prove Deligne's theorem for the case of a hypersurface of odd dimension $n$ in $\mathbb{P}_{\mathbb{F}_q}^{n+1}$ (provided that one knows Lefschetz theory well enough). That was the case that Deligne originally considered which convinced him that he can go all the way\footnote{He said that in an interview on the occasion of the Abel prize reward in 2013.}. We will come back to this in the end of Section 6.

\section{Overview of Lefschetz theory}

This section is a overreview of the Lefschetz theory results that we will use in Section 8. 

The first two subsections may be regarded as an independent review of the subject and are organized as follows: we first state the Lefschetz results on $\mathbb{C}$, then transpose these results into $l$-adic geometry and finally state the consequences for the higher direct images of $\mathbb{Q}_l$. The results will be carried into $l$-adic geometry by analogy. The standard reference for the proofs is SGA 7 (see exact references in [7], 5.13). Sketches of some of the results appear in [2], Chapters 31, 32. 

In the third subsection we actually check assumption (b) of the Main Lemma for a certain sheaf $\mathcal{F}_0$ (defined later). In the next section we will define this sheaf and verify that it also satisfies assumptions (a) and (c). The obtained result will be used in the final step of the proof in Section 8.  

\subsection{Local Lefschetz theory}

\textbf{Lefschetz results on $\mathbb{C}$} are as follows.

Let $D=\{z| \ |z| < 1\}$ be the unit disk, $D^*=D \setminus \{0\}$ and $f:X \to D$ a morphism of \textit{complex manifolds} (or, more generally, \textit{analytic spaces}). 

Assume that:

\textbf{(a)} $X$ is purely of dimension $n+1$ (and nonsingular in the case of analytic spaces).

\textbf{(b)} $f$ is proper.

\textbf{(c)} $f$ is smooth on the complement of the point $x$ of the \textit{special (singular) fiber} $X_0=f^{-1}(0)$.

\textbf{(d)} $x$ is a \textit{nondegenerate double point} of $f$.

Let $t \neq 0$ in $D$ and $X_t=f^{-1}(t)$ a \textit{general fiber}. To this data we associate:

\textbf{($\alpha$)} The \textit{specialization morphisms} $$sp: H^i(X_0, \mathbb{Z}) \to H^i(X_t, \mathbb{Z})$$ defined as the composition arrow in $$H^i(X_0, \mathbb{Z}) \overset{\sim} \leftarrow H^i(X, \mathbb{Z})\to H^i(X_t, \mathbb{Z}).$$

\textbf{($\beta$)} The \textit{monodromy transformations} $$T: H^i(X_t, \mathbb{Z}) \to H^i(X_t, \mathbb{Z}),$$ which describe the effect on the singular cycles of $X_t$ of "rotating $t$ around $0$". This is an action on $H^i(X_t, \mathbb{Z})$, seen as the stalk of $R^i f_* \mathbb{Z}|D^*$ (see Proposition 2 of Subsection 2.3) at $t$, of the positive generator of $\pi_1(D^*, t)$. 

We define the \textit{vanishing cycle} $$\delta \in H^n (X_t, \mathbb{Z})$$ as the unique (up to sign) generator of $$H^n(X_0, \mathbb{Z})^{\perp} \subset H^n (X_t, \mathbb{Z})$$ under the pairing induced by Poincare duality. 

Lefschetz theory describes ($\alpha$) and ($\beta$) in terms of $\delta$. For $i \neq n, \ n+1$ we have $$H^i(X_0, \mathbb{Z}) \overset{\sim} \to H^i(X_t, \mathbb{Z}) \ (i \neq n, \ n+1)$$ and for $i=n, \ n+1$ there is an exact sequence $$0 \to H^n(X_0, \mathbb{Z}) \to H^n(X_t, \mathbb{Z}) \overset{x \to (x, \delta)} \to \mathbb{Z} \to H^{n+1}(X_0, \mathbb{Z}) \to H^{n+1}(X_t, \mathbb{Z}) \to 0$$ (where $(\cdot, \cdot)$ is the duality pairing). For $i \neq n$, the monodromy $T$ is the identity and for $i=n$ we have $$Tx=x \pm (x, \delta)\delta.$$ 
The sign $\pm$ and the values of $T\delta$ and $(\delta, \delta)$ are as follows. 

\begin{center}
\begin{tabular}{ c c c c c c}
 \textbf{Table 1} & $n \mod 4$ & 0 & 1 & 2 & 3 \\ &
 $Tx=x\pm(x, \delta)\delta$ & - & - & + & + \\ &  
$(\delta, \delta)$ & 2 & 0 & -2 & 0 \\ & $T\delta$ & $-\delta$ & $\delta$ & $-\delta$ & $\delta$    
\end{tabular}
\end{center}
The monodromy transformation preserves the intersection form\footnote{It is a symplectic transvection for $n$ odd and the symmetric ortogonal for $n$ even.} $Tr(x \cup y)$ on $H^n(X_t, \mathbb{Z})$.

\textbf{$l$-adic Lefschetz results} are obtained as follows.

Replace the disk $D$ with the \textit{spectrum of a Henselian discrete valuation ring $A$ with an algebraically closed residue field}. Let $S$ be the spectrum, $\eta$ its generic point (spectrum of the field of fractions of $A$), $s$ the closed point (spectrum of the residue field). The role of $t$ is played by the geometric generic point $\bar \eta$ (spectrum of the closure of the field of fractions of $A$). Let $f: X \to S$ be a morphism of schemes. Assume that:

\textbf{(a)} $X$ is regular purely of dimension $n+1$.

\textbf{(b)} $f$ is proper.

\textbf{(c)} $f$ is smooth on the complement of the point $x$ of the special (singular) fiber $X_s$.

\textbf{(d)} $x$ is an \textit{ordinary double point}\footnote{A point $P$ of a variety $X$ of dimension $d$ is called an \textit{ordinary double point} if for the completion of the local ring $\mathcal{\hat O}_{X, P}$ one has $\mathcal{\hat O}_{X, P} \simeq k[[t_1, \cdots t_{d+1}]]/(Q(t_1, \cdots, t_{d+1}))$ where $Q(t_1, \cdots, t_{d+1})$ is a nondegenerate quadratic form.}.

Let $l$ be a prime number different from the characteristic of the residue field of $A$. Denote by $X_{\bar \eta}$ the generic geometric fiber. To this data we associate:

\textbf{($\alpha$)} The \textit{specialization morphism} $$sp: H^i(X_s, \mathbb{Q}_l) \overset{\sim} \leftarrow H^i(X, \mathbb{Q}_l) \to H^i(X_{\bar \eta}, \mathbb{Q}_l).$$ 

\textbf{($\beta$)} The \textit{local monodromy} action of the inertia group $I=Gal(\bar \eta/ \eta)$ on $H^i(X_{\bar \eta}, \mathbb{Q}_l)$: $$I=Gal(\bar \eta/\eta) \to GL(H^i(X_{\bar \eta}, \mathbb{Q}_l)).$$ Note that ($\alpha$) and ($\beta$) completely determine the sheaf $R^i f_*\mathbb{Q}_l$ on $S$. 

Let $n=2m$ for $n$ even and $n=2m+1$ for $n$ odd. We define the \textit{vanishing cycle}  $$\delta \in H^n(X_{\bar \eta}, \mathbb{Q}_l)(m)$$ as the unique (up to sign) generator of $$H^n(X_s, \mathbb{Q}_l)^{\perp} \subset H^n(X_{\bar \eta}, \mathbb{Q}_l)(m)$$ under the pairing induced by Poincare duality. 

Lefschetz theory describes ($\alpha$) and ($\beta$) in terms of $\delta$. For $i \neq n, \ n+1$ we have $$H^i(X_s, \mathbb{Q}_l) \overset{\sim} \to H^i(X_{\bar \eta}, \mathbb{Q}_l) \ (i \neq n, \ n+1)$$  and for $i=n, \ n+1$ there is an exact sequence $$0 \to H^n(X_s, \mathbb{Q}_l) \to H^n(X_{\bar \eta}, \mathbb{Q}_l) \overset{x \to Tr(x \cup \delta)} \to \mathbb{Q}_l(m-n) \to H^{n+1}(X_s, \mathbb{Q}_l) \to H^{n+1}(X_{\bar \eta}, \mathbb{Q}_l) \to 0.$$

The local monodromy action is trivial for $i \neq n$. For $i=n$ it is described as follows.

\textbf{(A) n odd} There is a canonical homomorphism $$t_l: I \to \mathbb{Z}_l(1),$$ and the action of $\sigma \in I$ is given by $$x \to x\pm t_l(\sigma)(x, \delta)\delta.$$

\textbf{(B) n even} We will not need this case. If $p \neq 2$, there exists a unique character of order two $$\varepsilon: I \to \{\pm1\},$$ and we have 
\begin{center}
\begin{tabular}{ c c c }
 $\sigma x=x$ & if & $\varepsilon(\sigma)=1$ \\ $\sigma x=x\pm(x, \delta)\delta$ & if & $\varepsilon(\sigma)=-1$    
\end{tabular}
\end{center}
The signs $\pm$ in (A) and (B) are the same as in Table 1.

\begin{fact}
The results above imply the following information about $R^i f_* \mathbb{Q}_l$.

\textbf{(a)} If $\delta \neq 0$:

\ \ \ \ 1) For $i \neq n$ the sheaf $R^if_*\mathbb{Q}_l$ is constant.

\ \ \ \ 2) Let $j$ be the inclusion of $\eta$ in $S$. We have $$R^if_*\mathbb{Q}_l=j_* j^* R^if_*\mathbb{Q}_l \ \footnote{Where by $"="$ we mean that the canonical morphism $R^if_*\mathbb{Q}_l \to j_* j^* R^if_*\mathbb{Q}_l$ is an isomorphism.}.$$

\textbf{(b)} If $\delta=0$ ( since $(\delta, \delta)=\pm 2$ for $n$ even, this can only happen for $n$ odd):

\ \ \ \ 1) For $i \neq n+1$ the sheaf $R^if_*\mathbb{Q}_l$ is constant.

\ \ \ \ 2) Let $\mathbb{Q}_l(m-n)_s$ be the sheaf $\mathbb{Q}_l(m-n)$ on $\{s\}$ extended by zero on $S$. Then there is an exact sequence $$0 \to \mathbb{Q}_l(m-n)_s \to R^{n+1}f_* \mathbb{Q}_l \to j_* j^*R^{n+1}f_* \mathbb{Q}_l \to 0,$$ where $j_* j^*R^{n+1}f_* \mathbb{Q}_l$ is a constant sheaf.
\end{fact}

\subsection{Global Lefschetz theory}

\textbf{Lefschetz results on $\mathbb{C}$} are as follows.

Let $\mathbb{P}$ be a projective space of dimension $\geq 1$ and $ \mathbb{\widecheck P}$ the dual projective space. The points of $ \mathbb{\widecheck P}$ parameterize the hyperplanes of $\mathbb{P}$ and we denote by $H_t$ the hyperplane defined by $t \in \mathbb{\widecheck P}$. If $A$ is  a linear subspace of codimension 2 in $\mathbb{P}$, hyperplanes containing $A$ are parameterized by the points of the projective line $D \subset \mathbb{\widecheck P}$, the \textit{dual} of $A$. We say that the set of hyperplanes $\{ H_t \}_{t \in D}$ forms the \textit{pencil with axis} $A$. 

Let $X \subset \mathbb{P}$ be a connected nonsingular projective variety of dimension $n+1$ and $\tilde X \subset X \times D$ be the set of pairs $(x, t)$ such that $x \in H_t$. Projections to the first and second coordinates form the diagram $$\begin{tikzcd}
X \arrow[r, leftarrow, "\pi"]
& \tilde X \arrow[d, "f"]\\
& D
\end{tikzcd} \ \ \textbf{(8)}$$ Denote by $X_t=X \cap H_t$ the fiber of $f$ at $t \in D$ (it is a hyperplane section of $X$). \textbf{For $A$ general enough we have:}

\textbf{(A)} $A$ is transverse to $X$ and $\tilde X$ is the blowing up of $X$ along $A \cap X$. In particular, $\tilde X$ is nonsingular.

\textbf{(B)} There is a finite subset $S$ of $D$ and for each $s \in S$ a point $x_s \in X_s$ such that $f$ is smooth outside the $\{x_s \}_{s \in S}$.

\textbf{(C)} All the $\{x_s \}_{s \in S}$ are critical nondegenerate points of $f$. 

So we see that \textbf{for each} $s \in S$ the \textbf{local Lefschetz theory} of the previous subsection \textbf{applies to a small disk $D_s$ around $s$ and $f^{-1}(D_s)$.}

Let $U=D \setminus S$, $u \in U$ and $\{ \gamma_s \}_{s \in S}$ be disjoint loops that start from $u$ and turn once around $s$. Then $\{ \gamma_s \}_{s \in S}$ generate the fundamental group $\pi_1(U, u)$. As in the local case, we have a \textit{monodromy action} $$\pi_1(U, u) \to Gl(H^i(X_u, \mathbb{Z})),$$ where we see $H^i(X_u, \mathbb{Z})$ as the stalk of $R^if_*\mathbb{Z}| U$ at $u$. According to the local theory, for each $s \in S$ we have a vanishing cycle $\delta_s \in H^n(X_u, \mathbb{Z})$. For $i \neq n$, the action of $\pi_1(U, u)$ on $H^i(X_u, \mathbb{Z})$ is trivial and for $i=n$ it "deforms the cohomology by vanishing cycles" and we have $$\gamma_s x=x \pm (x, \delta_s)\delta_s \ \ \textbf{(The Picard-Lefschetz formula).}$$

We define the \textit{vanishing part} $E$ of the cohomology as a subspace $$E \subset H^n(X_u, \mathbb{Q})$$ generated by all the $\{ \delta_s \}_{s \in S}$. It is easy to prove the following proposition.  

\begin{fact}
$E$ is stable under the action of the group $\pi_1(U, u)$ (the monodromy group). The ortogonal $E^{\perp}$ of $E$ (for the intersection form $Tr(x \cup y)$) is the space of invariants of $\pi_1(U, u)$ in $H^n(X_u, \mathbb{Q})$. 
\end{fact}

\begin{proof}
Both statements follow from the Picard-Lefschetz formula. If we substitute $\delta_{s'}$ for $x$, we will have $$\gamma_s \delta_{s'}=\delta_{s'} \pm (\delta_{s'}, \delta_s)\delta_s$$ that implies the first statement and $$\gamma_s x-x= \pm (x, \delta_s)\delta_s$$ implies the second
(because $\gamma_s$ generate the monodromy group).
\end{proof}

\newpage

The proof of the following proposition is a nice geometric argument. 

\begin{fact}
The vanishing cycles $\pm \delta_s$ are conjugate (up to sign)\footnote{That is, given $s, \ s' \in S$, there exists a $\sigma \in \pi_1(U, u)$ such that $\sigma \delta_s=\pm \delta_s$.} under the action of $\pi_1(U, u)$. 
\end{fact}

\begin{proof}
We define \textit{the dual variety} $\widecheck X$ of $X$ as the set of $t \in \mathbb{\widecheck P}$ such that $H_t$ is tangent to $X$ (that is, $X_t$ is singular or $X \subset H_t$). Then $\widecheck X$ is irreducible. Let $Y \subset X \times \mathbb{\widecheck P}$ be the space of pairs $(x, t)$ such that $x \in H_t$, then we have a diagram $$\begin{tikzcd}
X \arrow[r, leftarrow]
& Y \arrow[d, "g"]\\
& \mathbb{\widecheck P}
\end{tikzcd}$$ The fiber of $g$ at $t \in \mathbb{\widecheck P}$ is the hyperplane section $X_t= X \cap H_t$ of $X$ and $g$ is smooth on the complement of the inverse image of $\widecheck X$. 

Replacing $\mathbb{\widecheck P}$ by the line $D \subset \mathbb{\widecheck P}$ and $Y$ by $g^{-1}(D)$ we retrieve the situation of (8). We have $S=D \cap \widecheck X$ and, by a result of Lefschetz, for $D$ general enough $$\pi_1(D \setminus S, u) \to \pi_1(\mathbb{\widecheck P} \setminus \widecheck X, u)$$ is surjective so it suffices to show that $\pm \delta_s$ are conjugate under $\pi_1(\mathbb{\widecheck P} \setminus \widecheck X)$. 

We will first show that $\{ \gamma_s \}_{s \in S}$ are conjugate to each other. For $x$ in the smooth locus (of codimension $1$) of $\widecheck X$, let $ch$ be the path from $t$ to $x$ in $\mathbb{\widecheck P} \setminus \widecheck X$ and $\gamma_x$ the loop that follows $ch$ until the neighborhood of $\widecheck X$, turns once around $\widecheck X$ and then returns to $t$ by $ch$. The loops $\gamma_x$ (for various $ch$) are conjugate to each other in $\pi_1(\mathbb{\widecheck P} \setminus \widecheck X)$. 

We can always join two points in the smooth locus of $\widecheck X$ by a path that does not leave the smooth locus since $\widecheck X$ is irreducible. Therefore, the conjugation class of $\gamma_x$ does not depend on $x$ and in particular $\{ \gamma_s \}_{s \in S}$ are conjugate to each other. 

Proposition follows from the Picard-Lefschetz formula since $$\gamma_s x-x= \pm (x, \delta_s)\delta_s$$ determines $\delta_s$ up to sign and for $\gamma \in \pi_1(\mathbb{\widecheck P} \setminus \widecheck X)$ we have $$(\gamma \gamma_s \gamma^{-1})x=\gamma(\gamma^{-1} x \pm (\gamma^{-1}x, \delta_s)\delta_s)=x \pm (x, \gamma \delta_s)\gamma \delta_s.$$
\end{proof}

\begin{cor}
The action of $\pi_1(U, u)$ on $E/(E \cap E^{\perp})$ is absolutely irreducible\footnote{The action on the $k$-vector space $V$ is called absolutely irreducible if the corresponding action on $V \otimes_k \mathbb{C}$ is irreducible.}.
\end{cor}

\begin{proof}
Let $F \subset E \otimes \mathbb{C}$ be the subspace stable under the monodromy. If $F \not\subset (E \cap E^{\perp}) \otimes \mathbb{C}$, there exists an $x \in F$ and $s \in S$ such that $(x, \delta_s) \neq 0$ and we have $$\gamma_s x-x=\pm (x, \delta_s)\delta_s \in F.$$ So $\delta_s \in F$ and by Proposition 8 all the $\{ \delta_s \}_{s \in S}$ lie in $F$, so $F=E$.
\end{proof}

\textbf{$l$-adic Lefschetz results} are obtained as follows. 

Let $\mathbb{P}$ be a projective space of dimension $>1$ over an algebraically closed field $k$ of characteristic $p$ and $X \subset \mathbb{P}$ a connected projective nonsingular variety of dimension $n+1$. For $A$ a linear subspace of codimension 2 we define $D$, the pencil $\{ H_t \}_{t \in D}$, $\tilde X$ and the diagram (8) as in the complex case. 

\begin{defn}$\{ H_t \}_{t \in D}$ form a \textit{Lefschetz pencil} of hyperplane sections for $X$ if:

\textbf{(A)} The axis $A$ is transverse to $X$. $\tilde X$ is obtained by blowing up $X$ along $A \cap X$ and is smooth. 

\textbf{(B)} There is a finite subset $S$ of $D$ and for each $s \in S$ a point $x_s \in X_s$ such that $f$ is smooth outside the $\{x_s \}_{s \in S}$.

\textbf{(C)} All the $\{ x_s \}_{s \in S}$ are ordinary double singular points of $f$. 
\end{defn}

Then we see that \textbf{for each $s \in S$ the local Lefschetz theory} of the previous subsection \textbf{applies to the spectrum $D_s$ of the Henselization of the local ring of $D$ at $s$ and to $\tilde X_{D_s}=\tilde X \times_D D_s$.} 

Unlike the complex case, however, \textbf{it might happen that no Lefschetz pencil for $X$ can be found.} However, we \textbf{can overcome this complication} by modifying the projective embedding $\iota_1: X \xhookrightarrow{} \mathbb{P}$ as follows.

Let $N$ be the dimension of $\mathbb{P}$, $r$ an integer $\geq 1$ and $\iota_{(r)}$ the embedding of $\mathbb{P}$ into the projective space of dimension ${N+r \choose N} -1$ such that its homogeneous coordinates are monomials of degree $r$ in the homogeneous coordinates of $\mathbb{P}$. The hyperplane sections of $\iota_{(r)}(\mathbb{P})$ correspond to hypersurfaces of degree $r$ of $\mathbb{P}$. Choose an $r \geq 2$ and replace $\iota_1$ by $\iota_r=\iota_{(r)} \circ \iota_1$. 

One can show that \textbf{in the new embedding any general enough pencil of hyperplane sections is  a Lefschetz pencil for $X$.} The proof follows from a careful consideration of the dual variety $\widecheck X$ by applying Bertini's theorem (see [2], 31.2 for the sketch and SGA 7, XVII 3.7. for complete proof).

We will now study the Lefschetz pencil for $X$, \textit{excluding the case $p=2$, n even}. The case of $n$ odd will suffice for our purposes. Let $U=D \setminus S$, $u \in U$ and $l \neq p$ a prime number. The local results of the previous subsection (see Proposition 6) imply that $R^n f_* \mathbb{Q}_l$ is \textit{tamely ramified} at each $s \in S$. We recall the definition of the \textit{tame fundamental group}.

Fix an algebraic closure $K^{al}$ of $K=k(D)$. Then we let $$\pi_1^{tame} (U, u)=Gal(K^{tame}/K)$$ where $K^{tame}$ is the composite of the subfields of $K^{al}$ that are unramified at all the primes corresponding to the points of $U$ and tamely ramified at those corresponding to the points of $S$. Note that $\pi_1^{tame} (U, u)$ contains a subgroup $I_s$ for every $s \in S$ and is generated by these subgroups. We will drop the upper script and denote this group by $\pi_1 (U, u)$ as before.

The Lefschetz results are transposed from the complex case by replacing the fundamental group $\pi_1(U, u)$ with the tame fundamental group. The algebraic situation is similar to the  situation over $\mathbb{C}$ and the translation is done by standard arguments\footnote{The tame fundamental group of $U$ is a quotient of the profinite completion of the corresponding fundamental groups over $\mathbb{C}$ (lifting to characteristic $0$ of the tame coverings and the Riemann existence theorem).}. The proofs of the analogs of Proposition 7 and Corollary 3 are similar and in the proof of Proposition 8 the surjectivity result of Lefschetz for the map between fundamental groups is replaced by a theorem of Bertiny (see [2], 32.3 for the sketch and [24] for complete proof). We summarize the results in the following proposition. 
\newpage
\begin{fact}
Under the assumptions above we have the following results:

\textbf{(a)} \textit{If the vanishing cycles are nonzero}: 

\ \ \ \ 1) For $i \neq n$ the sheaf $R^if_*\mathbb{Q}_l$ is constant.

\ \ \ \ 2) Let $j$ be the inclusion of $U$ in $D$. We have $$R^nf_*\mathbb{Q}_l=j_* j^* R^nf_*\mathbb{Q}_l.$$

\ \ \ \ 3) Let $E \subset H^n(X_u, \mathbb{Q}_l)$ be the subspace of the cohomology generated by the vanishing cycles. This subspace is stable under $\pi_1(U, u)$ and $$E^{\perp}=H^n(X_u, \mathbb{Q}_l)^{\pi(U, u)}.$$ The action of $\pi(U, u)$ on $E/(E \cap E^{\perp})$ is absolutely irreducible and the image of $\pi_1(U, u)$ in $GL(E/(E \cap E^{\perp}))$ is topologically generated\footnote{They generate a dense subgroup.} by the maps $x \to x\pm(x, \delta_s)\delta_s \ (s \in S)$ (the $\pm$ sign is determined as in Table 1). 

\textbf{(b)} \textit{If the vanishing cycles are zero} (since $(\delta, \delta)=\pm 2$ for $n$ even, it can only happen for odd  $n=2m+1$):

\ \ \ \ 1) For $i \neq n+1$ the sheaf $R^if_*\mathbb{Q}_l$ is constant.

\ \ \ \ 2) We have an exact sequence $$0 \to \underset{s \in S} \oplus \mathbb{Q}_l(m-n)_s \to R^{n+1}f_*\mathbb{Q}_l \to \mathcal{F} \to 0$$ with $\mathcal{F}$ constant. 

\ \ \ \ 3) $E=0$.

Since all of the vanishing cycles are conjugate, there are no other possibilities.
\end{fact}

\subsection{The theorem of Kajdan and Margulis}
The subspace $E \cap E^{\perp}$ of $E$ is the kernel of the restriction to $E$ of the intersection form $Tr(x \cup y)$. Therefore, this form induces a bilinear nondegenerate form $$\psi: E/(E \cap E^{\perp}) \otimes E/ (E \cap E^{\perp}) \to \mathbb{Q}_l(-n),$$ preserved by the monodromy action, skew-symmetric for $n$ odd and symmetric for $n$ even. For $n$ odd, therefore, the monodromy action gives a map $$\rho: \pi_1(U, u) \to Sp(E/(E \cap E^{\perp}), \psi).$$

\begin{thrm14}
The image of $\rho$ is open.
\end{thrm14}

Intuitively, the theorem states that "\textit{the monodromy is big}". 

\begin{proof}
We will need some theory of Lie groups over $\mathbb{Q}_l$ (consult [21] or assume by analogy). The image of $\rho$ is closed so it is a Lie group over $\mathbb{Q}_l$ ($l$-adic analog of Cartan's theorem). Let $\mathfrak{L}$ be its Lie algebra\footnote{In the complex case we would have the Lie algebra of the Zariski closure of the monodromy group.}. Since there is an exponential map that sends any sufficiently small neighborhood of $0$ in a Lie algebra onto the neighborhood of $1$ of its Lie group, it suffices to show that $\mathfrak{L}$ equals $\mathfrak{sp}(E/(E \cap E^{\perp}), \psi)$.  

Consider the map $\log: Im(\rho) \to \mathfrak{L}$ (defined in the neighborhood of 1) and let $$\alpha_s: x \to x \pm \psi(x, \delta_s)\delta_s \ \ (s \in S).$$ Then $$\log(\alpha_s)=\log(1-(1-\alpha_s))=-\sum_n \frac{(1-\alpha_s)^n}{n}.$$
But $(1-\alpha_s)(x)=\pm \psi(x, \delta_s)\delta_s$. Because of skew-symmetry we have $\psi(\delta_s, \delta_s)=0$, so we see that $(1-\alpha_s)^2=0$ and $\mathfrak{L}$ contains all the endomorphisms $$N_s: x \to \psi(x, \delta_s)\delta_s \ \ (s \in S)$$ and is generated by them. Moreover, $E/(E \cap E^{\perp})$ is an absolutely irreducible representation of $\mathfrak{L}$. We obtain the theorem by applying the following lemma.
\end{proof}

\begin{lemma}
Let $V$ be a finite dimensional vector space over the field $k$ of characteristic $0$, $\psi$ a nondegenerate skew-symmetric form on a Lie subalgebra $\mathfrak{L}$ of the Lie algebra $\mathfrak{sp}(V, \psi)$. Assume: 

\ \ \ \ (i) $V$ is a simple representation of $\mathfrak{L}$. 

\ \ \ \ (ii) $\mathfrak{L}$ is generated by the family of endomorphisms of $V$ of the form $x \to \psi(x, \delta)\delta$. \\
Then $\mathfrak{L}=\mathfrak{sp}(V, \psi)$.
\end{lemma}

\begin{proof}
Assume that $V$ (and thus $\mathfrak{L}$) are nonzero. Let $W \subset V$ be the set of $\delta \in V$ such that $N(\delta): x \to \psi(x, \delta)\delta$ lies in $\mathfrak{L}$. Then we proceed as follows:

\textbf{(a)} $W$ is stable under homotheties (since $\mathfrak{L}$ is a vector subspace of $\mathfrak{gl}(V)$). 

\textbf{(b)} For $\delta \in W$, $\exp(\lambda N(\delta))$ (well-defined since $N(\delta)$ is nilpotent) is an automorphism of $(V, \psi, \mathfrak{L})$. $W$ is an invariant subspace for $\exp(\lambda N(\delta))$ and if $\delta', \ \delta'' \in W$, we have $$\exp(\lambda N(\delta'))\delta''=\delta''+\lambda \psi(\delta'', \delta')\delta' \in W.$$ It follows that if $\psi(\delta', \delta'') \neq 0$, then the vector subspace spanned by $\delta'$ and $\delta''$ lies in $W$. 

\textbf{(c)} From (b) we conclude that $W$ is the union of its maximal linear subspaces $W_{\alpha}$, pairwise orthogonal.  Each $W_{\alpha}$ is stable under the $N(\delta) \  (\delta \in W)$, so it is stable under $\mathfrak{L}$. Assumption (i) gives $W_{\alpha}=V$ and thus $\mathfrak{L}$ contains all the $N(\delta)$ for $\delta \in V$. 

But $\mathfrak{sp}(V, \psi)$ is generated by the $N(\delta) \ (\delta \in V)$ so our proof is complete.
\end{proof}

It is now easy to deduce:
\begin{thrm15}
Deligne's theorem is true for a hypersurface $X_0$ of odd dimension $n$ in $\mathbb{P}_{\mathbb{F}_q}^{n+1}$.
\end{thrm15}

\begin{proof}
We will sketch the proof. Let $X \simeq \mathbb{P}^n$ be the corresponding variety over $\mathbb{\bar F}_q$. It is enough to verify the claim of Deligne's theorem for $i=n$. This follows from Reduction 1 of Section 4 but in fact we do not need this here because one can show that $H^i( X, \mathbb{Q}_l)=0$ for odd $i \neq n$ and $H^i(X, \mathbb{Q}_l)=\mathbb{Q}_l(-\frac{i}{2})$ for even $i$ (see [2], 16.4), so we have $$Z(X_0, t)=\det(1-F^*t, H^n(X, \mathbb{Q}_l)/\prod_{i=0}^n(1-q^it)$$ (where $\det(1-F^*t, H^n(X, \mathbb{Q}_l))$ has rational coefficients) and it is enough to treat the case $i=n$. 

If we now vary $X_0$ within a Lefschetz pencil $f: \tilde X=\mathbb{P}^{n+1} \to D$ of hyperplane sections defined over $\mathbb{F}_q$ (it exists in a suitable projective embedding after taking a finite extension of $\mathbb{F}_q$ - see the proof in Section 8) one can verify that $E$ coincides with the whole $H^n(\tilde X, \mathbb{Q}_l)$ and applying the Main Lemma to the subsheaf $\mathcal{E}_0$ defining $E$ (see the next section) gives Deligne's theorem for all the hyperplanes of the pencil, in particular, for $X_0$. 
\end{proof}

\section{The rationality theorem}

Let $\mathbb{P}_0$ be a projective space of dimension $\geq 1$ over $\mathbb{F}_q$, $X_0 \subset \mathbb{P}_0$ a projective nonsingular variety, $A_0 \subset \mathbb{P}_0$ a linear subspace of codimension two, $D_0 \subset  \mathbb{\widecheck P}_0$ the dual projective line, $\mathbb{\bar F}_q$ the algebraic closure of $\mathbb{F}_q$ and $\mathbb{P}, X, A, D$ over $\mathbb{\bar F}_q$ obtained from $\mathbb{P}_0, X_0, A_0, D_0$ by extension of scalars. Diagram (8) of Subsection 6.2 comes from a similar diagram over $\mathbb{F}_q$:
$$\begin{tikzcd}
X_0 \arrow[r, leftarrow, "\pi_0"]
& \tilde X_0 \arrow[d, "f_0"]\\
& D_0
\end{tikzcd}$$

Assume that $X$ is connected of \textit{even} dimension $n+1=2m+2$ and that the pencil of hyperplane sections of $X$ defined by $D$ is a \textit{Lefschetz pencil}. The set $S$ of $t \in D$ such that $X_t$ is singular and defined over $\mathbb{F}_q$ comes from $S_0 \subset D_0$. We denote $U_0=D_0-S_0$ and $U=D-S$.

We will now define the sheaf $\mathcal{E}_0$ and prove that for $x \in |U_0|$ the action of $F_x$ on $\mathcal{E}_0/(\mathcal{E}_0 \cap \mathcal{E}_0^{\perp})$ is rational. Let $u \in U$. Since $E \subset H^n(X_u, \mathbb{Q}_l)$ is stable under $\pi_1(U, u)$, it is defined on $U$ by a local $\mathbb{Q}_l$-subsheaf $\mathcal{E}$ of $R^nf_*\mathbb{Q}_l$. Both sheaves are actually defined over $\mathbb{F}_q$. Indeed, $R^if_*\mathbb{Q}_l$ is an inverse image of the $\mathbb{Q}_l$-sheaf $R^if_{0*}\mathbb{Q}_l$ on $D_0$ and $\mathcal{E}$ is the inverse image of a local subsheaf $$\mathcal{E}_0 \subset R^nf_{0*}\mathbb{Q}_l$$ on $U_0$. The cup product defines a canonical skew-symmetric form $$\psi: R^nf_{0*}\mathbb{Q}_l \otimes R^nf_{0*}\mathbb{Q}_l \to \mathbb{Q}_l(-n).$$ Denote by $\mathcal{E}_0^{\perp}$ the ortogonal of $\mathcal{E}_0$ relative to $\psi$ on $R^nf_{0*}\mathbb{Q}_l | U_0$. We see that $\psi$ induces a nondegenerate (perfect) skew-symmetric pairing $$\psi: \mathcal{E}_0/(\mathcal{E}_0 \cap \mathcal{E}_0^{\perp}) \otimes \mathcal{E}_0/(\mathcal{E}_0 \cap \mathcal{E}_0^{\perp}) \to \mathbb{Q}_l(-n).$$ 

\begin{thrm16}
For all $x \in |U_0|$ the polynomial $\det (1-F_x^*t, \mathcal{E}_0/(\mathcal{E}_0 \cap \mathcal{E}_0^{\perp}))$ has rational coefficients. 
\end{thrm16}

\begin{cor}
Let $j_0$ be the inclusion of $U_0$ in $D_0$ and $j$ that of $U$ in $D$. The eigenvalues of $F^*$ acting on $H^1(D, j_*\mathcal{E}_0/(\mathcal{E}_0 \cap \mathcal{E}_0^{\perp}))$ are algebraic numbers all of which complex conjugates $\alpha$ satisfy $$q^{\frac{n+1}{2}-\frac{1}{2}} \leq |\alpha| \leq q^{\frac{n+1}{2}+\frac{1}{2}}.$$
\end{cor}

\begin{proof}
By the theorem of Kajdan and Margulis and the Rationality Theorem, the assumptions of the Main Lemma are satisfied for $(U_0, \mathcal{E}_0/(\mathcal{E}_0 \cap \mathcal{E}_0^{\perp}), \psi)$ and $\beta=n$ so we can apply Corollary 2 of Section 5. 
\end{proof}

\begin{proof}[Proof of the rationality theorem] 
We start with the following observation:

\begin{lemma}
Let $\mathcal{G}_0$ be a locally constant $\mathbb{Q}_l$-sheaf on $U_0$ such that its inverse image $\mathcal{G}$ on $U$ is a constant sheaf. Then there exist units $\alpha_i$ in $\mathbb{\bar Q}_l$ such that for each $x \in |U_0|$ we have $$\det(1-F_x^*t, \mathcal{G}_0)=\prod_i (1-\alpha_i^{\deg(x)} t).$$
\end{lemma}

\begin{proof}
Let $\varepsilon: U_o \to Spec \mathbb{F}_q$ be the canonical morphism. The assumption of the lemma implies that $\varepsilon^*\varepsilon_* \mathcal{G}_0 \to \mathcal{G}_0$ becomes an isomorphism once we base change to $\mathbb{\bar F}_q$ so $\mathcal{G}_0$ is the pullback of the sheaf $\varepsilon_{*} \mathcal{G}_0$.

This sheaf identifies with an \textit{$l$-adic representation} $G_0$ of $Gal(\mathbb{\bar F}_q/\mathbb{F}_q)$ (by Proposition 1) and we can take $\alpha_i$ as in $$\det(1-Ft, G_0)=\prod_i (1-\alpha_i t)$$ (since $F^*_x$ acts on $\mathcal{G}_0$ as $F^{-\deg(x)}$). 
\end{proof}

This lemma can be applied to $R^if_{0*} \mathbb{Q}_l \ (i \neq n)$, to $R^n f_{0*} \mathbb{Q}_l/\mathcal{E}_0$ and to $\mathcal{E}_0 \cap \mathcal{E}_0^{\perp}$(to see that the inverse images of the latter two sheaves are constant, use the equivalence of categories of Proposition 1 and the fact that the monodromy acts by deforming the cohomology by the vanishing cycles). 

Let $x \in |U_0|$ and consider the fiber $X_x=f_0^{-1}(x)$ as a variety over the finite field $k(x)$. Let $\bar x$ be a point of $U$ that lies above $x$, then $X_{\bar x}$ is obtained from $X_x$ by an extension of scalars of $k(x)$ to the algebraic closure $k(\bar x)=\mathbb{\bar F}_q$ and $H^i(X_{\bar x}, \mathbb{Q}_l)$ is the stalk of $R^if_*\mathbb{Q}_l$ at $\bar x$ (Proposition 2 of Subsection 2.3). The cohomological interpretation of the zeta function for $X_x$ gives $$Z(X_x, t)=\prod_i \det(1-F_x^* t, R^i f_{0*} \mathbb{Q}_l)^{(-1)^{i+1}}$$ and we use the filtrations $$0 \to \mathcal{E}_0 \to R^n f_{0*} \mathbb{Q}_l \to R^n f_{0*} \mathbb{Q}_l/\mathcal{E}_0 \to 0$$ and $$0 \to \mathcal{E}_0 \cap \mathcal{E}_0^{\perp} \to \mathcal{E}_0 \to \mathcal{E}_0/ \mathcal{E}_0 \cap \mathcal{E}_0^{\perp} \to 0$$ to represent $Z(X_x, t)$ as the product $Z(X_x, t)=Z_s(t) \cdot Z_b(t)$, where $$Z_s(t)=\det(1-F_x^* t, R^n f_{0*} \mathbb{Q}_l/\mathcal{E}_0)\det(1-F_x^* t, \mathcal{E}_0 \cap \mathcal{E}_0^{\perp}) \prod_{i \neq n} \det(1-F_x^* t, R^i f_{0*} \mathbb{Q}_l)^{(-1)^{i+1}}$$ is the the part with "\textit{small monodromy}" and $$Z_b(t)=\det(1-F_x^* t, \mathcal{E}_0/(\mathcal{E}_0 \cap \mathcal{E}_0^{\perp})$$ is the part with "\textit{big monodromy}". 

We want to show that $Z_b(t) \in \mathbb{Q}(t)$. Let  $\mathcal{F}_0=\mathcal{E}_0/(\mathcal{E}_0 \cap \mathcal{E}_0^{\perp})$, $\mathcal{F}=\mathcal{E}/(\mathcal{E} \cap \mathcal{E}^{\perp})$ and apply Lemma 11 to the factors of $Z^f$. There exist $l$-adic units $\{ \alpha_i \}_{1 \leq i 
\leq N}$ and $\{ \beta_j \}_{1 \leq j \leq M}$ in $\mathbb{\bar Q}_l$ such that for all $x \in |U_0|$ $$Z(X_x, t)=\frac{\underset{i} \prod (1-\alpha_i^{\deg(x)}t)}{\underset{j} \prod (1-\beta_j^{\deg(x)}t)} Z_b(t) \ \ \textbf{(9)}.$$ By rationality of the zeta function the right side lies in $\mathbb{Q}(t)$. After cancellation, we may assume that $\alpha_i \neq \beta_j$ for all $i$ and $j$. 

\textbf{It suffices to prove that} the polynomials $\underset{i} \prod (1-\alpha_i t)$ and $\underset{j} \prod (1-\beta_j t)$ have rational coefficients, that is, \textbf{the family of $\alpha_i$ (resp. the family of $\beta_j$) is defined over $\mathbb{Q}$.} The idea is to characterize the polynomials $\underset{j} \prod (1-\beta_j^{\deg(x)}t)$ and $\underset{i} \prod (1-\alpha_i^{\deg(x)}t)$ as the denominator and the numerator of (9) respectively. The problem here is that the factors coming from $Z_b(t)$ might cancel some of the $(1-\alpha_i^{\deg(x)}t)$. However, it can only happen to a very limited extent as we vary (9) over $x \in |U_0|$ because "the monodromy of $Z_b(t)$ is big" by the theorem of Kajdan and Margulis. We make this intuition precise in the following technical proposition.

\begin{fact}
Let $\{ \gamma_i \}_{1 \leq i \leq P}$ and $\{ \delta_j \}_{1 \leq j \leq Q}$ be two families of $l$-adic units of $\mathbb{\bar Q}_l$. Assume that $\gamma_i \neq \delta_j$ for all $i$ and $j$. If $K$ is a large enough set of integers $\neq 1$, and $L$ is a large enough nowhere dense subset of $|U_0|$, then, if $x \in |U_0|$ satisfies $k \nmid \deg(x)$ (for all $k \in K$) and $x \notin L$, the denominator of $$\det(1-F_x^*t, \mathcal{F}_0) \underset{i} \prod (1-\gamma_i^{\deg(x)}t)/ \underset{j} \prod (1-\delta_j^{\deg(x)}t) \ \ \textbf{(10)},$$ written in irreducible form, is $\underset{j} \prod (1-\delta_j^{\deg(x)}t)$. 
\end{fact}
Let me complete the proof of the theorem modulo Proposition 10 and then come back and prove Proposition 10. 
According to the following lemma, Proposition 10 provides an \textit{intrinsic description of the family of $\delta_j$} in terms of the family of rational fractions (10) for $x \in |U_0|$. 

\begin{lemma}
Let $K$ be a finite set of integers $\neq 1$ and $\{\delta_j \}_{1 \leq j \leq Q}$ and $\{ \varepsilon_j \}_{1 \leq j \leq Q}$ be two families of elements of a field. If, for all $n$ large enough and not divisible by any of the $k \in K$, the family of $\delta_j^n$ coincides with that of $\varepsilon_j^n$ (up to order), then the family of $\delta_j$ coincides with that of $\varepsilon_j$ (up to order). 
\end{lemma}

\begin{proof}
We will prove this by induction on $Q$ (with $Q$=1 the trivial case). The set of integers $n$ such that $\delta_Q^n=\varepsilon_j^n$ is an ideal $(n_j)$. Suppose that there exists no $j_0$ such that $\delta_Q=\varepsilon_{j_0}$. Then all the $n_j$ are distinct from $1$ and there exist arbitrarily large integers $n$, not divisible by any of the $n_j$ and any of the $k \in K$. For such an $n$ we have $\delta_Q^n \neq \varepsilon_j^n$, that is a contradiction. So there exists a $j_0$ such that $\delta_Q=\varepsilon_{j_0}$ and we conclude by applying the induction hypothesis to the families $\{ \delta_j \}_{j \neq Q}$ and $\{ \varepsilon_j \}_{j \neq j_0}$.
\end{proof}

By putting $\{ \gamma_i \}= \{ \alpha_i \}$ and $\{ \delta_i \}= \{ \beta_i \}$ in Proposition 10 we get \textit{an intrinsic characterization of the family of $\beta_j$} in terms of the family of rational functions $Z(X_x, t) \ (x \in |U_0|)$. Since $Z(X_x, t) \in \mathbb{Q}(t)$, \textbf{the family of $\beta_j$ is defined over $\mathbb{Q}$.} 

We will now give \textit{an intrinsic characterization of the family of $\alpha_i$}.
\begin{fact}
Let $\{ \gamma_i \}_{1 \leq i \leq P}$ and $\{ \delta_j \}_{1 \leq j \leq Q}$ be two families of $p$-adic units of $\mathbb{\bar Q}_l$ and let $R(t)=\underset{i} \prod (1-\gamma_i t)$, $S(t)=\underset{j} \prod (1-\delta_j t)$. Assume that for all $x \in |U_0|$ the polynomial $\underset{j} \prod (1-\delta_j^{\deg(x)} t)$ divides $$\underset{i} \prod (1-\gamma_i^{\deg(x)}t) \det(1-F_x^*t, \mathcal{F}_0).$$ Then $S(t)$ divides $R(t)$. 
\end{fact}

\begin{proof}
We remove from the families $\{ \gamma_i \}$ and $\{ \delta_j \}$ the pairs of common elements until they satisfy the assumption of Proposition 10 and apply Proposition 10. By assumption, the rational fractions (10) are polynomials. Therefore, no $\delta$ survives, which precisely means that $S(t)$ divides $R(t)$. 
\end{proof}

We have already established the rationality of $\underset{j} \prod (1-\beta_j^{\deg(x)}t)$, so $$\underset{i} \prod (1-\alpha_j^{\deg(x)}t) \det (1-F_x^* t, \mathcal{F}_0) \in \mathbb{Q}[t].$$ Proposition 11 provides an \textit{intrinsic description of the family of $\alpha_i$} in terms of this family of polynomials so \textbf{the family of $\alpha_i$ is also defined over $\mathbb{Q}$.}
\end{proof}

Our proof will be complete once we justify Proposition 10.
\begin{proof}[Proof of Proposition 10]
Let $u \in U$ and $\mathcal{F}_u$ the stalk of $\mathcal{F}$ at $u$. The \textit{arithmetic fundamental group} $\pi_1(U_0, u)$ is the extension of $\mathbb{\hat Z}=Gal( \mathbb{\bar F}/\mathbb{F})$ by the \textit{geometric fundamental group}: $$0 \to \pi_1(U, u) \to \pi_1(U_0, u) \to Gal(\mathbb{\bar F}_q/\mathbb{F}_q) \to 0.$$ The group $\pi_1(U_0, u)$ acts on $\mathcal{F}_u$ by symplectic similitudes $$\rho_0: \pi_1(U_0, u) \to GSp(\mathcal{F}_u, \psi).$$ 
This means that the action of $\pi_1(U_0, u)$ satisfies the property  $$\psi_u(gv, gw)=\mu(g)\psi(v, w), \ g \in \pi_1(U_0, u)$$ for a multiplicative character $\mu:Gl(\mathcal{F}_u) \to k^{\times}$ called the similitude multiplier.
This action restricts to the previously considered (in Subsection 6.3) monodromy representation of the geometric fundamental group: $$\rho: \pi_1(U, u) \to Sp(\mathcal{F}_u, \psi).$$
Since $\psi$ has values in $\mathbb{Q}_l(-n)$, it follows that the product of the canonical projection to $\mathbb{\hat Z}$ and $\rho$ takes $\pi_1(U_0, u)$ to the subgroup $$H \subset \mathbb{\hat Z} \times GSp(\mathcal{F}_u, \psi)$$ defined by the equation\footnote{Since $q$ is an $l$-adic unit, $q^n \in \mathbb{Q}_l^*$ is defined for all $n \in \mathbb{\hat Z}$.} $$q^{-n}=\mu(g), n \in \mathbb{\hat Z}.$$  
Let $$\rho_1: \pi_1(U_0, u) \to H$$ be this map and denote by $H_1$ the image of $\rho_1$. We need the following lemmas.
\begin{lemma}
$H_1$ is open in $H$. 
\end{lemma}

\begin{proof}
Indeed $\mathbb{\pi}_1(U_0, u)$ projects onto $\mathbb{\hat Z}$ and the image of $\pi_1(U, u)=Ker(\pi_1(U_0, u) \to \mathbb{\hat Z})$ in \\ $Sp(\mathcal{F}_u, \psi)=Ker(H \to \mathbb{\hat Z})$ is open by the theorem of Kajdan and Margulis.   
\end{proof}

\begin{lemma}
For $\delta \in \mathbb{\bar Q}_l$ an $l$-adic unit, the set $Z$ of $(n, g) \in H_1$ such that $\delta^n$ is an eigenvalue of $g$ is closed of measure $0$ (with respect to the Haar measure on $\mathbb{\hat Z} \times GSp(\mathcal{F}_u, \psi)$). 
\end{lemma}

\begin{proof}
It is clear that $Z$ is closed. Fix an $n \in \mathbb{\hat Z}$ and let $GSp_n$ be the set of $g \in GSp(\mathcal{F}_u, \psi)$ such that $\mu(g)=q^{-n}$, $Z_n=Z \cap GSp_n$. Then $Z_n$ is a closed proper algebraic subvariety of the inverse image in $H_1$ of $n$ so it has measure $0$. Therefore, we have shown that "fiberwise over $\mathbb{\hat Z}$" the subset $Z$ has measure $0$ and the claim follows by applying Fubini to the projection $H_1 \to \mathbb{\hat Z}$. 
\end{proof}

We can now complete the proof of Proposition 10. For each $i$ and $j$, the set of integers $n$ such that $\gamma_i^n=\delta_j^n$ is the set of multiples of a fixed integer $n_{ij}$ (we do not exclude $n_{ij}=0$). By assumption, $n_{ij} \neq 1$. 

According to Lemma 14 and the Chebotarev's density theorem, the set of $x \in |U_0|$ such that $\beta_j^{\deg(x)}$ is an eigenvalue of $F_x^*$ acting on $\mathcal{F}_0$ is nowhere dense. But then we can take for $K$ the set of all the $n_{ij}$ and for $L$ the set of $x$ as above so the proposition is proved. 
\end{proof}

\section{Completion of the proof of the Weil conjectures}
We are now ready to complete the proof of Deligne's theorem (and the Weil conjectures). After the reductions of Section 4 it remains to prove the following theorem. 

\begin{dt}
Let $X_0$ be a nonsingular absolutely irreducible projective variety of even dimension $d$ over $\mathbb{F}_q$, $X$ the corresponding variety over $\mathbb{\bar F}_q$ and $\alpha$ an eigenvalue of $F^*$ acting on $H^d(X, \mathbb{Q}_l)$. Then $\alpha$ is an algebraic number all of which complex conjugates, still denoted $\alpha$, satisfy $$q^{\frac{d}{2}-\frac{1}{2}} \leq |\alpha| \leq q^{\frac{d}{2}+\frac{1}{2}} \ \ \textbf{(11)}.$$ 
\end{dt}

\begin{proof}
We will prove this by induction on $d$ (always assumed even). The case $d=0$ is trivial so we assume from now on $d \geq 2$ and let $d=n+1=2m+2$.

In a suitable projective embedding $i: X \to \mathbb{P}$, $X$ admits a Lefschetz pencil of hyperplane sections (see Subsection 6.2). By Proposition 5 of Section 4 we may replace $\mathbb{F}_q$ by a finite extension so we may assume that the pencil is defined over $\mathbb{F}_q$. 

Therefore, we have a projective embedding $X_0 \to \mathbb{P}_0$ and a subspace $A_0 \subset \mathbb{P}_0$ of codimension two that defines the Lefschetz pencil. By taking a new extension of scalars we can assume that (in the notation of Sections 6 and 7): 

\textbf{(a)} The points of $S$ are defined over $\mathbb{F}_q$. 

\textbf{(b)} The vanishing cycles for $x_s \ (s \in S)$ are defined over $\mathbb{F}_q$ (since only $\pm \delta$ is intrinsic, we can only define them in quadratic extensions).

\textbf{(c)} There exists a rational point $u_0 \in U_0$. We choose a point $u$ over $u_0$ as the base point of $U$. 

\textbf{(d)} $X_{u_0}=f_0^{-1}(u_0)$ admits a smooth hyperplane section $Y_0$ defined over $\mathbb{F_q}$. We let $Y=Y_0 \otimes_{\mathbb{F}_q} \mathbb{\bar F}_q$.

Since $\tilde X$ is obtained from $X$ by blowing up along a smooth of dimension two subvariety $A \cap X$, we have an embedding $$H^i(X, \mathbb{Q}_l) \xhookrightarrow{} H^i(\tilde X, \mathbb{Q}_l)$$ (in fact, $H^i(\tilde X, \mathbb{Q}_l)= H^i(X, \mathbb{Q}_l) \oplus H^{i-2} (A \cap X, \mathbb{Q}_l)(-1))$by the Thom isomorphism theorem. See a, rather technical, proof in [2], 33.2). \textbf{So it suffices to prove (11) for the eigenvalues $\alpha$ of $F^*$ acting on $H^d(\tilde X, \mathbb{Q}_l)$}.

From the Leray spectral sequence for $f$ $$E_2^{pq}=H^p(D, R^q f_*\mathbb{Q}_l) \Rightarrow H^{p+q}(\tilde X, \mathbb{Q}_l)$$ we see that \textbf{it suffices to prove (11) for the eigenvalues of $F^*$ acting on $E_2^{pq}$ for $p+q=d=n+1$:} $$\textbf{(A)} \  E_2^{2, n-1}=H^2(D, R^{n-1} f_*\mathbb{Q}_l), \ \textbf{(B)} \ E_2^{0, n+1}=H^0(D, R^{n+1} f_*\mathbb{Q}_l) \ and \ \textbf{(C)} \ E_2^{1, n}=H^1(D, R^{n} f_*\mathbb{Q}_l).$$ One can suspect that case (C) is the troublesome one since the first cohomology group of curves is harder to deal with in general. We will therefore first treat (A) and (B). 

$\textbf{(A)} \  E_2^{2, n-1}=H^2(D, R^{n-1} f_*\mathbb{Q}_l).$ By Proposition 8, $R^{n-1}f_*\mathbb{Q}_l$ is a constant sheaf on $D$. By the properties of higher direct images (Proposition 2 of Subsection 2.3) and the cohomology groups of $D \simeq \mathbb{P}^1$ (Proposition 3 of Subsection 2.6) we see that $$H^2(D, R^{n-1} f_*\mathbb{Q}_l) \simeq (R^{n-1}f_* \mathbb{Q}_l)_u(-1)=H^{n-1}(X_u, \mathbb{Q}_l)(-1).$$ Consider the long exact sequence in cohomology $$\cdots \to H_c^{n-1} (X_u \setminus Y, \mathbb{Q}_l) \to H^{n-1}(X_u, \mathbb{Q}_l) \to H(Y, \mathbb{Q}_l) \to \cdots$$ induced by the the short exact sequence\footnote{The notations here are the same as in Corollary 4 of Section 7 and footnote 16 and this sequence is precisely the one that appears in Milne 8.15 (the derivation is easy and goes by considering the induced sequence of stalks).} $$0 \to j_! j^* \mathbb{Q}_l \to \mathbb{Q}_l \to i_* i^* \mathbb{Q}_l \to 0.$$

By Poincare duality (second form) and the Cohomological dimension property (b) of $l$-adic cohomology (Subsection 2.4) we have $$H_c^{n-1} (X_u \setminus Y, \mathbb{Q}_l) \simeq H^{n+1} (X_u \setminus Y, \mathbb{Q}_l\widecheck )=0$$ (since $X_u \setminus Y$ is affine of dimension $n$), so the map $H^{n-1}(X_u, \mathbb{Q}_l) \to H^{n-1}(Y, \mathbb{Q}_l)$ is injective and we can apply the induction hypothesis to $Y_0$.

$\textbf{(B)} \ E_2^{0, n+1}=H^0(D, R^{n+1} f_*\mathbb{Q}_l).$ \textbf{If the vanishing cycles are nonzero}, $R^{n+1}f_* \mathbb{Q}_l$ is constant (see Proposition 8) and $$H^0(D, R^{n+1} f_*\mathbb{Q}_l) \simeq (R^{n+1} f_*\mathbb{Q}_l)_u=H^{n+1}(X_u, \mathbb{Q}_l).$$ Lefschetz theorem (see Subsection 2.4) shows that the Gysin map $$H^{n-1}(Y, \mathbb{Q}_l)(-1) \to H^{n+1}(X_u, \mathbb{Q}_l)$$ is surjective and we apply the induction hypothesis to $Y_0$ again. 

\begin{remark}
By expanding the proof of the weak Lefschetz theorem one can see that in fact the argument here is dual to that of (A).
\end{remark}

\textbf{If the vanishing cycles are zero}, the exact sequence of Proposition 8 (b) gives the following exact sequence $$\underset{s \in S} \oplus \mathbb{Q}_l (m-n)_s \to E_2^{0, n+1} \to H^{n+1}(X_u, \mathbb{Q}_l).$$ The eigenvalues of $F$ acting on $\mathbb{Q}_l (m-n)_s$ are $q^{n-m}=q^{d/2}$ and for $H^{n+1}(X_u, \mathbb{Q}_l)$ everything is as above. 

It remains to handle the case of $F^*$ acting on $E_2^{1, n}=H^1(D, R^{n} f_*\mathbb{Q}_l)$.

$\textbf{(C)} \ E_2^{1, n}=H^1(D, R^{n} f_*\mathbb{Q}_l).$ 
If the vanishing cycles are zero, $R^nf_* \mathbb{Q}_l$ is constant (Proposition 8 (b)) and $E_2^{1, n}=0$ so \textbf{we may assume that the vanishing cycles are nonzero}. 

We can filter $R^nf_* \mathbb{Q}_l=j_*j^* R^n f_* \mathbb{Q}_l$ (Proposition 8) by the subsheafs $j_*\mathcal{E}$ and $j_*(\mathcal{E} \cap \mathcal{E}^{\perp})$. \textbf{If the vanishing cycles $\delta$ are not in $\mathcal{E} \cap \mathcal{E}^{\perp}$}(or, rather,  in $E \cap E^{\perp}$) we have exact sequences\footnote{To see that the following sheaves are constant, use the equivalence of categories stated in Proposition 1, the facts that the monodromy acts by deforming the cohomology by the vanishing cycles and that the inverse/direct images of constant sheaves via $j$ are constant (the last claim is true here because $R^nf_* \mathbb{Q}_l=j_*j^* R^n f_* \mathbb{Q}_l$ but false in general).}: $$0 \to j_*\mathcal{E} \to R^n f_* \mathbb{Q}_l \to constant \ sheaf \ j_*(R^n f_* \mathbb{Q}_l/ \mathcal{E}) \to 0$$ $$0 \to \ constant \ sheaf \ j_*(\mathcal{E} \cap \mathcal{E}^{\perp})) \to j_*\mathcal{E} \to j_*(\mathcal{E}/(\mathcal{E} \cap \mathcal{E}^{\perp})) \to 0$$
\textbf{If the vanishing cycles $\delta$ are in $\mathcal{E} \cap \mathcal{E}^{\perp}$} (in fact, one can later show that this does not happen - see remark after the proof), we have $\mathcal{E} \subset \mathcal{E}^{\perp}$. Let $\mathcal{F}=R^nf_*\mathbb{Q}_l/j_* \mathcal{E}^{\perp}$. Then we have exact sequences: $$ 0 \to \  constant \ sheaf \ j_*\mathcal{E}^{\perp} \to R^n f_* \mathbb{Q}_l \to \mathcal{F} \to 0$$ $$ 0 \to \mathcal{F} \to \ constant \ sheaf \ j_*j^* \mathcal{F} \to \underset{s \in S} \oplus \mathbb{Q}_l(n-m)_s \to 0$$ 

\textbf{In the first case} the long exact sequences in cohomology give $$H^1(D, j_* \mathcal{E}) \to H^1(D, R^n f_* \mathbb{Q}_l) \to 0$$ $$0 \to H^1(D, j_*\mathcal{E}) \to H^1(D, j_*(\mathcal{E}/(\mathcal{E} \cap \mathcal{E}^{\perp})))$$ and we apply Corollary 4 to $H^1(D, j_* \mathcal{E})$.

\textbf{In the second case}, they give $$0 \to H^1(D, R^n f_*\mathbb{Q}_l) \to H^1(D, \mathcal{F})$$ $$ \underset{s \in S} \oplus \mathbb{Q}_l(n-m) \to H^1(D, \mathcal{F}) \to 0$$ and we remark that $F^*$ acts on $\mathbb{Q}_l(n-m)$ by multiplication by $q^{n-m}=q^{d/2}$.
\end{proof}
Therefore, we have finally proved Deligne's theorem and the Weil conjectures are fully justified.

\begin{remark}
If we had the hard Lefschetz theorem (most people believed it to be necessary for the proof but Deligne managed to derive at as a consequence that we will state in the next section) the proof for (C) would be much shorter.  Indeed, it implies that $\mathcal{E} \cap \mathcal{E}^{\perp}$ is zero and that $R^n f_* \mathbb{Q}_l$ is the direct sum of $j_*\mathcal{E}$ and a constant sheaf. The first cohomology group of a constant sheaf on $\mathbb{P}^1$ is zero and it would suffice to apply Corollary 4.
\end{remark}

\section{Consequences}

In this section we state some of the bright applications of Deligne's theorem, prove a generalization of the Hasse-Weil bound to nonsingular complete intersections and briefly discuss and provide references for other results (also see [7], par.8 and an overview of N.Katz [25]). 

\subsection{Consequences in Number theory}

The fact that there are many spectacular applications of Deligne's theorem in Number theory comes as no surprize. Let me introduce a few:

\textbf{(A)} One can estimate $\# Y_0 (\mathbb{F}_q)$ for a variety $Y_0$ over $\mathbb{F}_q$ provided that the corresponding variety $Y$ over $\mathbb{\bar F}_q$ has a simple enough cohomological structure. For example, for $Y_0$ a nonsingular $n$-dimensional hypersurface of degree $d$ we get $$|\# Y_0(\mathbb{F}_q)- \sum_{i=0}^n q^i| \leq \left( \frac{(d-1)^{n+2}+(-1)^{n+2}(d-1)}{d} \right)q^{\frac{n}{2}}.$$ More generally, one can prove the following.

\begin{thrm17}
Let $X_0 \subset \mathbb{P}_0^{n+r}$ be a nonsingular complete intersection over $\mathbb{F}_q$ of dimension $n$ and of multidegree $(d_1, \cdots d_r)$. Let $b'$ be the $n$-th Betti number of the complex nonsingular complete intersection with the same dimension and multidegree. Put $b=b'$ for $n$ odd and $b=b'-1$ for $n$ even. Then $$|\# X_0(\mathbb{F}_q)-\# \mathbb{P}^n(\mathbb{F}_q)| \leq b q^{n/2}.$$
\end{thrm17}

\begin{proof}
Let $X$ over $\mathbb{\bar F}_q$ be the variety obtained from $X_0$ by extension of scalars and for each $i$ let $\mathbb{Q}_l \eta^i$ be the line in $H^{2i}(X, \mathbb{Q}_l)$ generated by the $i$-th cup power of the cohomology class of a hyperplane section. $F^*$ acts on $\mathbb{Q}_l \eta^i$ by multiplication by $q^i$. 

The cohomology of $X$ is the direct sum of the $\{ \mathbb{Q}_l \eta^i \}_{0 \leq i \leq n}$ and the \textit{primitive part} of $H^n (X, \mathbb{Q}_l)$ of dimension $b$. Lemma 1 of Subsection 2.5 implies that there exist $b$ algebraic numbers $\alpha_j$, the eigenvalues of $F^*$ acting on this primitive part of the cohomology, such that 

$$\# X_0(\mathbb{F}_q)=\sum_{i=0}^n q^i+(-1)^n \sum_j \alpha_j.$$ According to Deligne's theorem, $|\alpha_j|=q^{n/2}$ and $$|\# X_0(\mathbb{F}_q)-\# \mathbb{P}^n(\mathbb{F}_q)|=|\# X_0(\mathbb{F}_q)-\sum_{i=0}^n q^i|= |\sum_j \alpha_j| \leq \sum_j |\alpha_j|=b q^{n/2}.$$
\end{proof}
This result was recently generalized to singular complete intersections (see [26]).

\textbf{(B)} Deligne's theorem also provides consequences for \textit{the theory of modular forms}. In particular, it implies the following.

\begin{thrm18}
The Ramunajan's $\tau$-function given by the Fourier coefficients $\tau(n)$ of the cusp form $\Delta(z)$ of weight 12: $$\Delta(z)=\sum_{n \geq 1} \tau(n) q^n =q \prod_{n \geq 1} (1-q^n)^{24}$$ with $q=e^{2 \pi i z}$ for any prime $p$ satisfies $$|\tau(p)| \leq 2p^{\frac{11}{2}}.$$
\end{thrm18}

One can also prove the following generalization to other modular forms.

\begin{thrm19}
Let $N$ be an integer $\geq 1$, $\varepsilon: (\mathbb{Z}/N)^* \to \mathbb{C}^*$ a character, $k$ an integer $\geq 2$ and $f$ a holomorphic modular form on $\Gamma_0(N)$ of weight $k$ and with character $\varepsilon$ : $f$ is  a holomorphic function on the the Poincare half-plane $X$ such that for $\left(\begin{array}{cc} a & b\\ c & d \end{array}\right) \in SL(2, \mathbb{Z}), $ with $c \equiv 0 \ (N)$ we have $$f \left(\frac{az+b}{cz+d} \right)=\varepsilon(a)^{-1}(cz+d)^kf(z).$$ We assume that $f$ is cuspidal and primitive, in particular $f$ is an eigenvector of the Hecke operators $T_p \ (p \nmid N)$. Let $f=\sum_{n=1}^{\infty} a_nq^n$ with $q=e^{2\pi i z}$ (and $a_1=1$). Then for $p$ prime not dividing $N$ $$|a_p| \leq 2p^{\frac{k-1}{2}}.$$ In other words, the roots of the equation $$T^2-a_p T+\varepsilon(p)p^{k-1}$$ are of absolute value $p^{\frac{k-1}{2}}$.
\end{thrm19}

Indeed, the roots turn out to be the eigenvalues of the Frobenius endomorphism acting on $H^{k-1}(X, \mathbb{Q}_l)$ of a nonsingular projective variety $X$ of dimension $k-1$ defined over $\mathbb{F}_p$. That was originally proved (under restrictive assumptions) by Deligne in [27]. 

I highly recommend reading Deligne's appendix to [28] for proofs of both of the results. Let me also note that Deligne and Serre managed to prove the second conjecture for $k=1$ as well but the proof is quite different. 

\textbf{(C)} Deligne's theorem can also be used to estimate \textit{exponential sums in several variables}. In particular:

\textbf{(a)} Let $f$ be a polynomial of $n$ variables over $\mathbb{F}_q$ of degree $d$ prime to $p$ such that its homogenious part $f_d$ defines a nonsingular projective hypersurface. Then $$|\sum_{x_i \in \mathbb{F}_q} \exp \left( \frac{2 \pi i}{p} f(x_1, \cdots, x_n) \right)| \leq (d-1)^n q^{n/2}.$$

\textbf{(b)} We have the following bound for \textit{multiple Kloosterman sums}: $$|\sum_{x_i \in \mathbb{F}_q^{\times}} \exp \left( \frac{2 \pi i}{p} \left(x_1+ \cdots +x_n +\frac{1}{x_1 \cdots x_n}\right) \right)| \leq (n+1) q^{n/2}.$$

Both (a) and (b) come from the following fact (originally suggested by Bombiery).

\begin{thrm20}
Let $Q$ be a polynomial in $n$ variables of degree $d$ over $\mathbb{F}_q$, $Q_d$ its homogeneous part of degree $d$ of $Q$ and $\psi: \mathbb{F}_q \to \mathbb{C}^*$ an additive nontrivial character on $\mathbb{F}_q$. Assume that: 

(i) $d$ is coprime to the characteristic of $\mathbb{F}_q$.

(ii) The hypersurface $H_0$ in $\mathbb{P}_{\mathbb{F}_q}^{n-1}$ defined by $Q_d$ is nonsingular. 

Then $$|\sum_{x_i \in \mathbb{F}_q} \psi(Q(x_1, \cdots, x_n))| \leq (d-1)^n q^{n/2}.$$
\end{thrm20}
As might be expected, the main idea is to represent the sum $$\sum_{x_i \in \mathbb{F}_q} \psi(Q(x_1, \cdots, x_n))$$ as the left hand side of formula (7) of Subsection 2.5 for a certain sheaf $\mathcal{F}$.
The proof of this theorem is given in [7], 8.4-8.13.

\textbf{(D)} Consider the following setup. Let $X$ be a smooth projective surface defined over $\mathbb{F}_q$ and let $\pi: X \to \mathbb{P}^1$ be a regular map which fibers $X_t$ are elliptic curves except for a finite number of $t \in \mathbb{\bar F}_q$. For $t \in \mathbb{P}^1(\mathbb{F}_q)$ let $$\# X_t(\mathbb{F}_q)=q-e(t)+1.$$ If we choose a nonsingular fiber $X_t$ then Deligne's theorem implies that $$|2e(t)| < 2 \sqrt{q}.$$ One may wonder if there are similar estimates for $\underset{t} \sum e(t)$, $\underset{t} \sum e(t)^2$ or other sums of this sort. An approach already discussed above, namely, reinterpreting the sum as the trace of the Frobenius endomorphism on a curve, then studying the monodromy action and applying Deligne's theorem, yields the following results:

\begin{fact}
Let $e(t)$ be defined as above. Then:

(a)If the $j$ invariant of the family is not constant then $$|\sum_t e(t)| \leq (\beta_2(X)-2)q,$$ where $\beta_2(X)$ is the second Betti number of $X$.

(b)If If the $j$ invariant of the family is not constant then $$\sum_t e(t)^2=q^2+O(q^{\frac{3}{2}}).$$

(c) Let $S$ be the set of $t \in \mathbb{\bar F}_q$ such that $X_t$ is singular. If the sets $S$ and $\{s-u, \ s \in S \}$ are disjoint, then $$\sum_t e(t)e(t+u)=O(q^{\frac{3}{2}}).$$
\end{fact}

See [2], Chapter 34 for an overview of (D) and the book of Katz [29] for a more comprehensive study of applications of Etale cohomology to estimations of various sums. 

What makes the correspondence between number-theoretic estimates and the monodromy of sheaves truly fascinating is that it can actually be used in both directions. For example, N.Katz proved the following fact (see [30]) using the estimates of Davenport and Lewis on the number of solutions of polynomials over finite fields:

\begin{fact}
Let $f(X, Y) \in \mathbb{C}[X, Y]$ be a polynomial and suppose that for indeterminates $a, \ b, \ c$ the complete nonsingular model of the affine curve $$f(X, Y)+aX+bY+c=0$$ over $\mathbb{C}(a, b, c)$ has genus $g \geq 1$. Then for any nonempty Zariski open set $S \subset \mathbb{A}_{\mathbb{C}}^3$ over which the complete nonsingular model extents (in a particular way) to a morphism $f: \mathbb{C} \to S$, the fundamental group of $S$ acts absolutely irreducibly on a general stalk of $R^1 f_* \mathbb{Q}$ (higher direct image for the complex topology). 
\end{fact}

\subsection{Other consequences}

The following few paragraphs follow section III of [25] (consult it for more details).

Let me say a few words about the consequences of Deligne's work on the cohomology and the theory of weights. As foreseen by Grothendieck in the early 1960's with his "\textit{yoga of weights}", the proof of the Weil conjectures for curves over finite fields would have important consequences for the \textit{cohomological structure of varieties over $\mathbb{C}$}. 

The idea here is that any reasonable algebro-geometric structure over $\mathbb{C}$ is actually defined over some subring of $\mathbb{C}$ that is finitely generated over $\mathbb{Z}$ and \textit{reducing modulo a maximal ideal} $m$ of that ring we get to the finite field situation with the Frobenius endomorphism $F(m)$. This Frobenius operates on the $l$-adic cohomology that is precisely the singular cohomology with $\mathbb{Q}_l$ coefficients of the original complex situation. 
This means that we get a previously unsuspected structure on the cohomology of a complex algebraic variety, the so-called "\textit{weight filtration}". 

In the late 1960's and early 1970's Deligne has developed the \textit{mixed Hodge theory} - a complete theory of weight filtration of complex algebraic varieties (see [31] for an introduction). One of the applications of this theory is the \textit{global monodromy} of families of projective smooth varieties (given a smooth projective map $X \to S$ of smooth complex varieties for any $s \in S$ the monodromy representation of $\pi_1(S, s)$ on all the $H^i(X_s, \mathbb{C})$ is \textit{completely reducible}). 

By means of an extremely ingenious argument (see the proof in [14]) drawing upon the theory of $L$-functions and ideas of Hadamard and de la Valee Poussin proof of the \textit{prime number theorem} Deligne established an $l$-adic analog of the \textit{complete reducibility theorem} for $l$-adic local systems in characteristic $p$ over etale open subsets of $\mathbb{P}^1$ provided that all the fibers of the local system satisfy the Riemann Hypothesis. Therefore, Deligne's theorem implies that we can apply this result to the local system coming from a Lefschetz pencil for a projective nonsingular variety over a finite field. The resulting complete reducibility implies the following.

\begin{thrm21}
Let $X$ be a smooth projective variety over an algebraically closed field $k$, $\mathcal{L}$ be an invertible ample sheaf on $X$ and $\eta=c_1(\mathcal{L}) \in H^2(X, \mathbb{Q}_l)$ be the cohomology class of a hyperplane section. We assume that $X$ is purely of dimension $n$. Then for each $i \geq 0$ the cup product by $\eta^i$: $$\eta^i \cup: H^{n-i}(X, \mathbb{Q}_l) \to H^{n+i}(X, \mathbb{Q}_l)$$ is an isomorphism. 
\end{thrm21}

I would like to conclude this paper with a curious idea (which might have already occurred to the reader) that has lead to the development of the so-called \textit{$\mathbb{F}_1$-geometry}. We have noted that the Hasse-Weil zeta function $\zeta_Y(s)$ for $Y=Spec(\mathbb{Z})$ gives the Riemann zeta function. Therefore, if we could see $Spec(\mathbb{Z})$ as a "scheme of finite type over $\mathbb{F}_1$" one could hope to reprove Deligne's theorem for $Spec(\mathbb{Z})$ and verify the Riemann hypothesis. 

Therefore, different approaches to $\mathbb{F}_1$ geometry deem to create an object that has some conjectured useful properties (such as the described above). Unfortunately, nobody has yet formulated how $Spec(\mathbb{Z})$ can be of finite type nor what a "\textit{smooth $\mathbb{F}_1$-variety}" should mean.

As far as I am aware, none of the existing $\mathbb{F}_1$-theories has so far been used to prove any important external results, nor can we call any of the approaches actually well-developed but there are many people who have high hopes that such a theory can be established and used to prove both the Riemann hypothesis and other strong number-theoretic statements such as the ABC-conjecture. I refer an interested reader to [32] for a thorough review of the existing approaches to $\mathbb{F}_1$-geometry. 

\section*{Acknowledgements}
I would like to use the chance to thank my scientific advisor Michael Finkelberg for his support and encouragement, Pavel Solomatin for his useful remarks and Richard Griffon for being the first one to inspire my interest in the theme.

\end{document}